\documentclass[10pt]{amsart}
\theoremstyle{plain}
\usepackage{amsmath, amssymb, amscd,graphicx}
\newtheorem{lem}{Lemma}[section]
\newtheorem{prop}[lem]{Proposition}
\newtheorem{thrm}[lem]{Theorem}
\newtheorem{cor}[lem]{Corollary}

\newtheorem{defn}[lem]{Definition}

\newtheorem{tm}{Theorem}
\newtheorem*{question}{Question}
\newtheorem{conjecture}{Conjecture}

\newtheorem{crrr}[tm]{Corollary}

\newcommand{\Ozsvath}{{Ozsv{\'a}th  }}
\newcommand{\Szabo}{{Szab{\'o} }}

\newcommand{\Z}{{\mathbb{Z}}}

\newcommand{\hfhat}{{{\widehat{HF}}}}

\newcommand{\spi}{{{\mathbf s}}}

\newcommand{\spinc}{{{\mathrm{Spin}^c}}}
\newcommand{\gr}{{{\mathrm{gr} }}}

\newcommand{\bfx}{{{{\bf x}}}}

\newcommand{\ts}{{{\thinspace}}}

\renewcommand{\Ozsvath}{Ozsv{\'a}th }
\renewcommand{\Szabo}{Szab{\'o} }
\newcommand{\CF}{\widehat{CF}}
\newcommand{\HFK}{\widehat{HFK}}
\newcommand{\T}{\mathbb{T}}
\newcommand{\s}{\mathfrak{s}}
\newcommand{\x}{{\bf x}}
\newcommand{\y}{{\bf y}}
\newcommand{\Spinc}{\text{Spin}^c}
\renewcommand{\spinc}{\text{spin}^c}
\newcommand{\DBar}{\overline{\Delta}}

\renewcommand{\spi}{\mathfrak{s}}
\newcommand{\alphas}{\mbox{\boldmath$\alpha$}}
\newcommand{\betas}{\mbox{\boldmath$\beta$}}

\title[Lens space surgeries and L-space homology spheres]{Lens space surgeries and L-space homology spheres }
\author{Jacob Rasmussen}
\address{Princeton University Dept. of Mathematics, Princeton, NJ
  08544 \newline \phantom{Xx}  and DPMMS, University of Cambridge, UK}
\email{jrasmus@math.princeton.edu}
\thanks{The author was supported by NSF grant DMS-0603940 and a Sloan Fellowship.}

\begin{document}

\begin{abstract}
We describe necessary and sufficient conditions for a knot in an
L-space to have an L-space homology sphere surgery. We use these conditions
 to reformulate a conjecture of Berge about which knots in
\(S^3\) admit lens space surgeries. 
\end{abstract}

\maketitle

\section{Introduction}

Let \(K\) be a knot in \(S^3\). If \(r/s\) Dehn surgery on \(K\)
yields the lens space \(L(p,q)\), we say that \(K\) admits a lens
space surgery, and that the lens space is realized by surgery on
\(K\). It is a longstanding problem to determine which
\(K\subset S^3\) admit lens space surgeries. The question was first
raised by L. Moser \cite{Moser}, who showed that all torus knots admit
lens space surgeries. Later, many other examples were found
\cite{RolfsenLens}, \cite{FSLens}, culminating with the work of Berge
\cite{Berge}, who gave a conjecturally complete list of such knots. 

There has also been considerable work on the converse problem of
finding necessary conditions for a knot to admit a lens space
surgery. Perhaps the most important result in this direction is the
Cyclic Surgery Theorem of Culler, Gordon, Luecke, and Shalen
\cite{CyclicS}, which implies (among other things) that if \(K\) admits a
lens space surgery, then either \(K\) is a torus knot or the surgery
coefficient is an integer. More recently, \Ozsvath and \Szabo have
used Heegaard Floer homology to give strong constraints on the knot
Floer homology of a knot admitting a lens space surgery \cite{OSLens}.
In conjunction with work of Ni \cite{YiThesis}, their work implies that any
such \(K\) must be fibred.  

The argument in \cite{OSLens} relies on the fact that the Heegaard
Floer homology of a lens space is as small as possible. A
three-manifold with this property is called an {\it L-space}. More
formally, a rational homology sphere \(Z\) is an L-space if and only if
\(\hfhat(Z) \cong \Z^p\), where \(p=|H_1(Z)|\). The main theorem
of \cite{OSLens} gives necessary and sufficient conditions for a knot
in an L-space homology sphere to admit an L-space surgery. 

In this
paper, we consider the converse problem. Given a knot \(K\) in an
L-space \(Z\), when does \(K\) 
admit a surgery which is an L-space homology sphere (or LHS, for
short)? As it turns out, the answer to this question
 depends mainly on the genus of \(K\). If surgery on
 \(K\)   yields a homology sphere,  then \(K\) must generate
\(H_1(Z)\). (We call such knots {\it primitive}.) Thus
 \(K\) will  not bound a Seifert surface in
 \(Z\) unless \(Z\) is a homology sphere.
 Nevertheless, there is still a natural notion of the genus \(g(K)\):
if \(Z_0\) is the complement of a regular neighborhood of \(K\), we
define \(g(K)\) to be the minimal genus of a surface \(\Sigma \subset
(Z_0, \partial Z_0)\) whose boundary defines a nontrivial class in
\(H_1(\partial Z_0)\). We have

\begin{tm}
\label{Thrm:LHSCrit}
Let \(K\subset Z\) be a knot in an L-space, and suppose that some integer
surgery on \(K\) yields a homology sphere \(Y\). If
 \(g(K) < (|H_1(Z)|+1)/2 \), then \(Y\) is an L-space,
while if  \(g(K) > (|H_1(Z)|+1)/2 \), then \(Y\) is not an L-space.
\end{tm}
\noindent
There is also a precise description of what happens  when
\(g(K) = (|H_1(Z)|+1)/2\), but this is  more complicated to state,
so we postpone it to a later section.

The theorem has several antecedents. Most notably, 
a  similar theorem was proved  
by Hedden in \cite{HeddenLens}, using a different method. Also, the
second half of the theorem was originally proved (in the context of monopole
Floer homology) by Kronheimer, Mrowka, Ozsv{\'a}th, and \Szabo
\cite{KMOS}.

If \(K\subset S^3\) is a knot with a lens space surgery
\(L(p,q)\), there is  a dual knot \(\widetilde{K} \subset L(p,q)\) which admits
an \(S^3\) surgery. One of Berge's key insights is that it is
often better to study \(\widetilde{K}\) than \(K\). Indeed, in all of
Berge's examples, this dual knot has a particularly nice form: it is
an example of what we will call a {\it simple knot} in a lens
space. For readers familiar with Heegaard Floer homology, these knots
are easy to describe: they are the knots obtained by placing two
basepoints inside the standard genus one Heegaard diagram of
\(L(p,q)\). We will give a more precise definition in section 
\ref{Sec:Knots}; for the moment, it is enough to know that there
is a unique simple knot in each
homology class in \(H_1(L(p,q)).\)

\begin{tm}
\label{Thrm:Simple}
Suppose \(K \subset L(p,q)\), and let \(K'\)
be the simple knot in the same homology class. If \(K\) admits an
integer LHS surgery, then so does \(K'\); in addition, either \(g(K) =
(p+1)/2\), or 
\(g(K)=g(K')\) and the two knots have isomorphic knot Floer homology. 
\end{tm}

The theorem suggests the following three-part approach to the Berge
conjecture ({\it c.f.} the similar program put forward by Baker,
Grigsby, and Hedden in \cite{BGH}).
First, determine all the simple knots in lens spaces which admit integer
LHS surgeries. Second, show that none of
the knots with \(g(K)=(p+1)/2\) which admit integer LHS surgeries 
actually yield \(S^3\). Finally, try to prove that if a simple knot
admits an integer LHS surgery, it is unique, 
 in the sense that it is the only knot in its homology class with that
knot Floer homology. 

The first step in this process can be reduced to
 a purely number-theoretic problem. In
section~\ref{Sec:Simple}, we describe an elementary
 algorithm for computing the knot
Floer homology of a simple knot in a lens space;
 by a theorem of Ni \cite{YiQHS}, this
determines its genus. In \cite{Berge}, Berge describes several
families of simple knots which are shown to have \(S^3\)
surgeries. More recently, Tange \cite{Tange} has found several
additional families of simple knots which have surgeries yielding the
Poincar{\'e} sphere (which is also an L-space). Based on computer
calculations of the genus function, we make the following 

\begin{conjecture}
\label{Conj:Num}
If \(K\) is a simple knot in \(L(p,q)\) which admits an integer
\(\Z\)HS surgery and has  \(g(K) < (p+1)/2\), then \(K\)
belongs to one of the families enumerated by Berge and Tange.
\end{conjecture}
\noindent 

By combining Theorem~\ref{Thrm:Simple} with an argument using the
Ozsv{\'a}th-Szab{\'o} \(d\)-invariant, it is not difficult to show
that Conjecture~\ref{Conj:Num} would imply the following {\it
  Realization Conjecture}:

\begin{conjecture}
\label{Conj:Real}
If \(L(p,q)\) is realized by integer surgery on a knot \(K\subset 
S^3\), then it is realized by integer surgery on a Berge knot. 
\end{conjecture}

The last step in this program seems considerably 
more difficult. The list of knots which are known to be determined by
their knot Floer homology is rather small: in \(S^3\), the unknot, the
trefoil, and the figure-eight knot are the only known examples. These
three knots are all distinguished by some geometrical property which
is reflected in their knot Floer homology: the unknot is the only knot
of genus zero, while the trefoil and figure-eight knots are the only
fibred knots of genus one. The Berge knots exhibit a similar
geometrical property --- they are genus minimizing in their homology
class. More  generally, we have 

\begin{tm}
\label{Thrm:MinGenus}
Suppose that \(Z\) is an L-space and that
 \(K \subset Z\) is a primitive knot with
 \(g(K)<(|H_1(Z)|+1)/2\). If \(K'\) is another knot
representing the same homology class as \(K\), then \(g(K') \geq
g(K)\). 
\end{tm}

To achieve the third step, it would be enough to show that if
\(K\) is a simple knot in \(L(p,q)\) with \(g(K)<(p+1)/2\), then \(K\) is
the {\it unique} genus minimizer in its homology class.
Interestingly, a theorem of Baker \cite{Baker} says that this is
true whenever  \(g(K) \leq  (p+1)/4\). As an application of Baker's
theorem, we have

\begin{crrr}
\label{Cor:L4}
If integer surgery on \(K \subset S^3\) yields  
\(L(4n+3,4)\), then \(K\) is the positive \((2,2n+1) \) torus knot. 
\end{crrr}

More generally, we can ask the following

\begin{question} Does a simple knot in \(L(p,q)\) minimize genus in
  its homology class? If so, is it the unique minimizer?  If not, what
   is the minimizer? 
\end{question}

The genus of some simple knots is quite large, so it seems rather bold
to imagine that this question has a positive answer. On the other hand,
a brief computer survey of \((1,1)\) knots in lens spaces failed to
produce any examples 
which had genus less than or equal to that of the corresponding
simple knot, so the problem is not without interest.

In a somewhat different direction, one can ask how  many different L-space
homology spheres exist. It is not hard  to see that the manifolds
obtained by repeatedly connected summing the Poincar{\'e} sphere (with
either orientation) with itself are  L-space homology spheres. Along
with \(S^3\), these are the only examples known at present. If
\(K \subset L(p,q)\) is in the same homology class as a Berge knot and
 satisfies \(g(K) = (p+1)/2\), then a theorem of Tange \cite{Tange2} 
shows that
 either \(K\) is a counterexample to the Berge conjecture or the
 manifold obtained by surgery on \(K\) is a new  L-space
 homology sphere. 


The remainder of this paper is organized as
follows. Sections~\ref{Sec:Knots} and \ref{Sec:HFK} are mostly review;
section~\ref{Sec:Knots} discusses the problem of when a knot in a
rational homology sphere admits a homology sphere surgery, and
section~\ref{Sec:HFK} recalls \Ozsvath and Szab{\'o}'s theory of knot
Floer homology for knots in a rational homology sphere \cite{OSFrac}. In
section~\ref{Sec:LHS}, we use the mapping cone formula from
\cite{OSFrac} to prove Theorem~\ref{Thrm:LHSCrit}, 
and in section~\ref{Sec:FoxBrody} we apply a theorem of Fox and Brody
\cite{Brody} to prove Theorems~\ref{Thrm:Simple} and
\ref{Thrm:MinGenus}. Finally, section~\ref{Sec:Simple} describes an
algorithm to compute the genus of a simple knot in a lens space and
gives some numerical evidence for Conjecture~\ref{Conj:Num}. 

The author would like to thank Matt Hedden for sharing his work in
 \cite{HeddenLens}, Motoo Tange for kindly providing a copy
of his preprint \cite{Tange},  and Ken Baker, Nathan Dunfield, Eli Grigsby,
  Zolt{\'a}n Szab{\'o}, and Dylan Thurston for helpful conversations.

\section {Knots in Rational Homology Spheres}
\label{Sec:Knots}

We begin by recalling some basic facts  about knots in rational
homology three-spheres. 
Suppose \(K \subset Z\) is an oriented knot in an oriented
rational homology sphere,
and let  \(Z_0\subset Z\) be the complement of a regular neighborhood
of \(K\). The orientation on \(K\) determines an oriented meridian
\(m \in H_1(\partial Z_0)\). We also choose an oriented longitude
\(\ell \in H_1(Z_0)\) with the property that \(\ell \cdot m = 1\) with
respect to the orientation on \(\partial Z_0 \) induced by \(Z_0\). 

\(Z_0\) is a three-manifold with torus boundary, so there is an
essential curve in \(\partial Z_0\) which bounds in \(Z_0\). Let 
\(\alpha = am + p \ell \in H_1(\partial Z_0)\) be a primitive homology
class represented by such a curve, oriented so that \(p > 0 \). 
\begin{defn}
The genus \(g(K)\) is the minimal genus of a properly embedded
orientable surface \(\Sigma \subset Z_0\) with \([\partial \Sigma] =
\alpha\). 
\end{defn}
\noindent
When \(K\) is null-homologous, this is just
the usual definition of the Seifert genus of \(K\). 

The
number \(p\) is well-defined and is the {\it order} of \(K\) in
\(H_1(Z_0)\). Replacing \(\ell\) by \(\ell+m\) has the effect of
replacing \(a\) by \(a-p\), so the value of
\(a\) mod \(p\) is also an invariant of \(K\). 
The quantity \(a/p \ \text{mod} \ 1\) is  the {\it
  self-linking number} \(K\cdot K\) of \(K\). More geometrically, it
may be defined as follows: the class \(p[K]\) is null-homologous, so
it bounds a Seifert surface \(\Sigma \subset Z\) with \(\Sigma \cap
\partial Z_0 = \alpha\). Then \(K \cdot K =
(\ell \cdot \Sigma)/p = \ell \cdot \alpha/p = a/p\). 
From this definition, it is not difficult to see that
the self-linking number depends only on the homology class of \(K\),
and that it is quadratic: \([nK] \cdot [nK] \equiv n^2 [K] \cdot [K] \
\text{mod} \ 1.\) 

An {\it integer surgery} on \(K\) is a manifold \(Z'\) obtained by Dehn
filling \(Z_0\) along the curve \(km + \ell\) for some \(k \in
\Z\). More geometrically, \(Z'\) is obtained by integer surgery on
\(K\) if and only if there is a cobordism with boundary \(-Z \cup
Z'\) obtained by attaching a two-handle to \(Z\times I \) along
\(K\). From this point of view, it is clear that the relation of being
an integer surgery is symmetric: if \(Z'\) is obtained by integer
surgery on \(K \subset Z\), then \(Z\) is obtained by integer surgery
on the dual knot \(\widetilde{K} \subset Z'\), where \(\widetilde{K}\) is the belt sphere of the
original two-handle. 

\begin{lem}
\label{Lem:ZHSSurg}
Let \(K \subset Z\) be a knot in a rational homology sphere. 
Then \(K\) has an integer sugery
which is a homology sphere if and only if \([K]\) generates \(H_1(Z)\)
(so that \(H_1(Z) \cong \Z/p\) for some \(p\))  
and its self-linking number \(a/p\) 
is congruent to \(\pm 1/p\) mod \(1\).  
\end{lem}

\begin{proof}
Consider the Mayer-Vietoris sequence for the decomposition \(Z=Z_0
\cup S^1\times D^2\):
\begin{equation*}
\begin{CD}
0 @>>> H_1(T^2) @>>> H_1(S^1) \oplus H_1(Z_0) @>>> H_1(Z) @>>>0
\end{CD}
\end{equation*}
and for the decomposition \(Z' = Z_0 \cup S^1 \times D^2\):
\begin{equation*}
\begin{CD}
@>>> H_2(Z') @>>> H_1(T^2) @>>> H_1(S^1) \oplus H_1(Z_0) @>>> H_1(Z') @>>>0.
\end{CD}
\end{equation*}
If \(Z'\) is a homology sphere, the second sequence tells us that 
\(H_1(T^2) \cong H_1(S^1) \oplus H_1(Z_0)\), so \(H_1(Z_0) \cong
\Z\). The same sequence also tells us that \(H_2(Z_0) \cong 0\). 
Next, we consider the long exact sequence of the pair \((Z_0, \partial
Z_0)\):
\begin{equation*}
\begin{CD}
@>>> H_2(Z_0,\partial Z_0) @>>> H_1(\partial Z_0) @>>> H_1(Z_0) @>>>
H_1(Z_0,\partial Z_0) @>>>0.
\end{CD}
\end{equation*}
The last group in the sequence is isomorphic to \(H^2(Z_0)\), which
vanishes by the universal coefficient theorem. It follows that \(H_1(
\partial Z_0)\) surjects onto \(H_1(Z_0)\), so the latter group is
generated by the images of \(m\) and \(\ell\). 

Returning to the first sequence, we consider the maps \(H_1(S^1) \to
H_1(Z)\) and \(H_1(Z_0) \to H_1(Z)\). The image of \(H_1(S^1)\) is
clearly generated by \([K]\), while the image of \(H_1(Z_0)\) is
generated by the image of \(m\), which is trivial, and the image of
\(\ell\), which is \([K]\). Since \(H_1(S^1) \oplus H_1(Z_0)\)
surjects onto \(H_1(Z)\), we conclude that \([K]\)  generates
\(H_1(Z)\). 

Conversely, suppose that \([K]\) generates \(H_1(Z)\). Then in the
first sequence, 
 \(H_1(S^1)\) surjects onto \(H_1(Z)\). From
 this, it is easy to see that  \(H_1(Z_0)\) must be torsion free, and
 thus isomorphic to \(  \Z\).
 
We now consider the map \(H_1(T^2) \to H_1(S^1\times D^2)
 \oplus H_1(Z_0)\) in
the second sequence. Let \(\beta = km + \ell\) be the image of \(\partial D^2\)
in \(H_1(T^2)\). Then the map \(H_1(T^2) \to H_1(S^1 \times D^2) \cong
\Z\)
 is given by \(x \mapsto x \cdot \beta\). Similarly, the map 
\(H_1(T^2) \to H_1(Z_0) \cong \Z \) is given by \(x \mapsto x \cdot \alpha\),
 where \(\alpha = am + p \ell\). Thus with respect to the basis
 \((m,\ell)\) on \(H_1(T^2)\), the map \( H_1(T^2) \to H_1(S^1) \oplus
 H_1(Z_0)\) is given by the matrix
\begin{equation*}
A=
\begin{bmatrix} -1 & k \\ -p & a
\end{bmatrix}
\end{equation*}
In order for the map to be an isomorphism, we must choose \(k\) so
that \(\det A = \pm 1\), which is possible if and only if \( a \equiv
\pm 1 \ \text{mod} \ p.\)
\end{proof}

If \(K \subset Z\) generates \(H_1(Z) \cong \Z/p\), we say that \(K\)
is a {\it primitive knot of order \(p\)} in \(Z\). 

\begin{lem} 
\label{Lem:LinkingForm}
Suppose \(K \subset Z\) is a primitive knot of order \(p\)
 in a rational
  homology sphere. Then the self-linking number \(a/p\) of \(K\) is
  characterized by
\begin{enumerate}
\item The set of manifolds obtained by integer surgery on \(K\) can be
  identified with the set of integers \(m \equiv -a \ (p)\), where the
  manifold \(K_m\) corresponding to \(m\) has \(H_1(K_m) \cong
  \Z/m\). 
\item Let \(x\) be a generator of \(H_1(Z_0)\cong \Z\), and let
  \(i_*(x)\) be its image in \(H_1(Z)\). Then \([K]= a x\).   
\end{enumerate}
\end{lem}

\begin{proof}
The first part follows easily from the proof of
Lemma~\ref{Lem:ZHSSurg}. For the second part, note that 
\(K\) is homologous to \(\ell\) in \(Z\). The image of \(\ell\) in
\(H_1(Z_0) \cong \Z\) is given by \(\ell \cdot \alpha = a\). 
\end{proof}

\subsection {Simple knots in lens spaces} 
\label{SubSec:SimpleKnot}
We now describe a family of
examples which will be particularly important in what follows.  The lens
space  \(L(p,q)\) can be decomposed as 
\(S^1 \times D^2_\alpha \cup S^1 \times D^2_\beta\),
so that if \(\alpha \subset T^2\) is the boundary of \(D^2_\alpha\)
and  \(\beta\) is the boundary of \(D^2_\beta\),
 there is a fundamental domain for \(T^2\) in which \(\alpha\)
is horizontal and \(\beta\) has slope \(p/q\). This decomposition
naturally gives rise to a Heegaard diagram for \(L(p,q)\), as
illustrated in Figure~\ref{Fig:SimpleKnot}. We orient \(L(p,q)\) so
that the orientation on \(S^1 \times D^2_\alpha\) is the
standard one, and the orientation on the other solid torus
 is reversed. (Note that with this convention, 
\(L(p,q)\)  is \(+p/q\) surgery on the unknot; this agrees with the
convention used by \Ozsvath and \Szabo, but is the opposite of the one
used in \cite{GompfStip}.)
\begin{figure}
\includegraphics{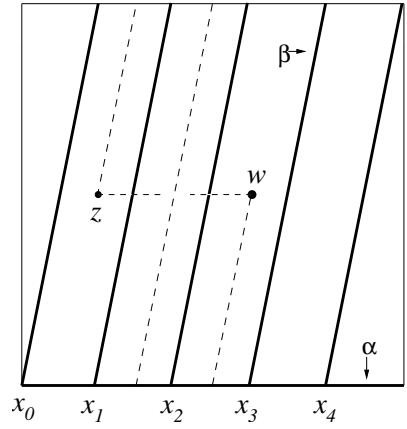}
\caption{\label{Fig:SimpleKnot} A Heegaard diagram for the lens
  space \(L(5,1)\), and the knot \(K(5,1,2)\) within it. The
  horizontal line segment has been pushed slightly into the alpha handlebody.}
\end{figure}

The disks \(A = \{0\} \times D_\alpha\) and \(B = \{0\} \times
D_\beta\) intersect at \(p\) points along their boundaries; these are
the places where \(\alpha\) intersects \(\beta\) in the Heegaard
diagram. We label these points \(x_0,x_1,\ldots,x_{p-1}\) in order of
their appearance on \(\alpha\), as shown in
Figure~\ref{Fig:SimpleKnot}. 

\begin{defn} \cite{Berge} The simple knot \( K(p,q,k) \subset L(p,q)\)
  is the oriented knot which is the union of an arc
  joining \(x_0\) to \(x_k\) in \(A\) with an arc joining \(x_k\) to \(x_0\)
  in \(B\). 
\end{defn}
\noindent
In the above definition, it is most convenient to take \(p>0\), and
view \(q\) and \(k\) as elements of \(\Z/p\). Note that 
by translating the fundamental domain of the Heegaard torus, we
 could just as well have used \(x_i\) and \(x_{i+k}\), for any \(i \in \Z/p\). 

To draw \(K(p,q,k)\) in the Heegaard diagram, we replace
the disks \(A\) and \(B\) by translates \(A' = \{x\} \times D_\alpha\)
and \(B' = \{y\} \times D_\beta\), so that \(x_{i}\) and
\(x_{i+k}\) are replaced by translates \({z} = x_i'\) and
\({w} = x_{i+k}'\). To get the knot, we join \({z}\) to \({
  w}\) by a horizontal segment in \(S^1 \times D^2_{\alpha}\) and
\({w}\) to \({z}\) by a segment of slope \(p/q\) in 
\(S^1 \times D^2_{\beta}\), as illustrated in
Figure~\ref{Fig:SimpleKnot}. Equivalently, as described in 
\cite{OS7} the knot \(K(p,q,k)\) is derived
from the doubly-pointed Heegaard diagram \((T^2, \alpha, \beta, {z},
{w})\).

\begin{lem}
We have the following relations among the \(K(p,q,k)\): 
\begin{enumerate}
\item \(K(p,q,-k)\) is the orientation-reverse of \(K(p,q,k)\).
\item \(K(p,-q,-k)\) is the mirror image of \(K(p,q,k)\) in
  \(\overline{L(p,q)} = L(p,-q)\). 
\item \(K(p,q,k) \cong K(p,q',kq')\), where \( qq' \equiv 1 \ (p)\).  
\end{enumerate}
\end{lem}

\begin{proof}
The first two identifications are elementary. For the third, observe
that the identification \(L(p,q) \cong L(p,q^{-1})\) can be obtained by 
exchanging  the roles of \(\alpha\) and \(\beta\) in the Heegaard
diagram. As we travel along
the (original) beta curve, we encounter the \(x_{i}\)'s in the
following order: \(x_0,x_{q},x_{2q},\ldots x_{(p-1)q}\). The point
\(x_k\) is in the \(q'k\)-th position in this list. 
\end{proof}

We would like to know when the knot \(K=K(p,q,k)\) admits a homology
sphere surgery. To determine its homology class, note that 
\(K\) is homotopic to an immersed curve in the Heegaard torus. The
 image of \([K]\) in \(H_1(S^1 \times D^2_{\beta})\) is given by
\([K] \cdot \beta = k\), so  \([K]= k [b]\), where \(b =
S^1 \times \{0\} \subset S^1 \times D_{\beta}\) is the core curve of
the beta handlebody. Thus \([K]\) generates \(H_1(L(p,q))\)
precisely when \(k\) is relatively prime to \(p\). 

To compute the self-linking number of \(K\), we observe that \([b]
\cdot [b] \equiv q'/p \ \text{mod} \ 1\). 
Thus \(K \cdot K \equiv k^2 q'/p \ \text{mod} \ 1\). Now if
\(k^2q' \equiv \pm 1 \ (p)\), then \(k\) must be relatively prime to
\(p\), so \(K(p,q,k)\) is a primitive knot in \(L(p,q)\). 
In summary, we have proved

\begin{lem}
\label{Lem:SimpleHS}
The knot \(K(p,q,k)\) has an integer surgery which is a homology
sphere if and only if  \(k^2 \equiv \pm q \ (p)\). 
\end{lem}

\section{Knot Floer homology}
\label{Sec:HFK}
In this section, we  briefly review the theory of knot Floer homology for
rationally null-homologous knots, as developed by \Ozsvath and
\Szabo in \cite{OSFrac}. With the exception of
Proposition~\ref{Prop:Chi}, all of this material may be found in
\cite{OSFrac} ({\it c.f} \cite{OS7}, \cite{thesis}.) 
  To keep things simple, we will focus on
the case where \(K\) is a primitive knot of order \(p\)
in a rational homology sphere \(Z\).

\subsection{Heegaard diagrams}
\label{SubSec:Diagrams}
Any knot \(K \subset Z\) can be represented by a {\it doubly pointed
Heegaard diagram} \((\Sigma, \alphas, \betas, {z}, {
  w})\), as illustrated in Figure~\ref{Fig:SimpleKnot}b. Here
\(\Sigma\) is a surface of genus \(g\), and \({\alphas} =
\{\alpha_1,\ldots, \alpha_g\}\) and 
\({\betas} = \{\beta_1,\ldots, \beta_g\}\) are two sets of {\it
  attaching circles} on \(\Sigma\). In other words, \(\alpha_1,\alpha_2,
\ldots \alpha_g\) are  embedded, disjoint, simple closed curves on
\(\Sigma\) which are linearly independent in \(H_1(\Sigma)\), and
similarly for the \(\beta_i\). The triple
 \((\Sigma, {\alphas}, {\betas})\) is a Heegaard diagram for \(Z\), {\it
   i.e.} \(\Sigma\) is a Heegaard surface for \(Z\) so that the
 \(\alpha_i\)'s bound compressing disks in one handlebody bounding
 \(\Sigma\), and the 
\(\beta_i\)'s bound compressing disks in the other.

 The knot \(K\) is specified by the two {\it basepoints} \({ z}\)
 and \({w}\) in \(\Sigma - {\alphas } - {\betas}\) by the
 following rule: we join \({ z}\) to \({ w}\) by an arc  in \(
 \Sigma\) which is disjoint from \({\alphas}\) and push it slightly
 into the alpha handlebody, Similarly, we join \({w}\) to \({
   z}\) by an arc in \(\Sigma\) which is disjoint from \({\betas}\)
 and push this arc slightly into the beta handlebody. 
 \(K\) is the union of these two arcs.
 
 Given such a doubly-pointed diagram, we can construct a Heegaard
 diagram for the complement of a regular neighborhood of \(K\) as
 follows. First, we remove small neighborhoods of \({ z}\) and
 \({ w}\) from \(\Sigma\). We then join the resulting boundaries by a
 tube to form a new surface \(\Sigma'\) of genus \(g+1\). 
 Finally, we  add an additional alpha circle \(\alpha_{g+1}\), which runs
 from \( z\) to \( w\) in \(\Sigma\), and then back over the
 tube. In the new diagram, the meridian of the knot \(K\) is
 represented by a small circle linking the tube. This process is
 illustrated for the knot \(K(5,1,2)\) in
 Figure~\ref{Fig:KnotComplement}.

 \begin{figure}
\includegraphics{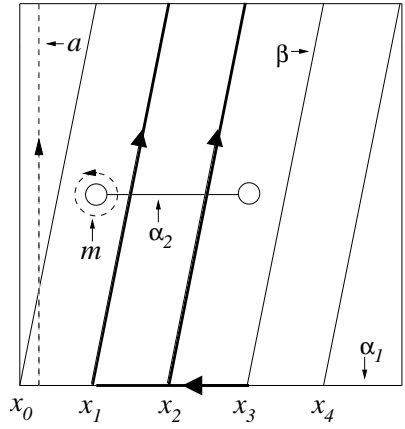}
 \caption{\label{Fig:KnotComplement} A genus two Heegaard diagram for
  the complement of \(K(5,1,2)\). The heavy
  lines show the one-chain \(\eta(x_3,x_1)\).}
 \end{figure}
 
 \subsection{Generators}
Given a doubly pointed Heegaard diagram \((\Sigma, {
  \alphas}, { \betas}, {z}, {w})\) which represents \(K\),
\Ozsvath and \Szabo construct a filtered chain complex
\(\widehat{CF}(K)\). This complex depends on the doubly pointed
Heegaard diagram,   but its filtered chain homotopy type is an 
invariant of \(K\). 

The generators of \(\CF(K)\) are easy to describe;
they consist of unordered \(g\)-tuples of intersection points \({\bf
  x} = \{x_1,x_2, \ldots x_g \}\) between the alpha and beta
curves, such that each alpha and beta curve is represented exactly
once. To be precise, each \(x_i\) is in \(\alpha_j \cap \beta_k\) for
some \(j\) and \(k\), and each \(\alpha_j\) and \(\beta_k\) contains
exactly one \(x_i\). More geometrically, the generators correspond to
the intersection points of two half-dimensional tori
\(\mathbb{T}_\alpha\), \(\mathbb{T}_\beta\) in the symmetric product
\(\text{Sym}^g \Sigma\). For this reason, the set of generators is
usually denoted by \(\mathbb{T}_\alpha \cap \mathbb{T}_\beta\). Each
generator has a \(\Z/2\) valued {\it homological grading}, which is
given by the sign of the corresponding intersection between
\(\mathbb{T}_\alpha\) and \(\mathbb{T}_\beta\).
\vskip0.05in
\noindent {\bf Example:} The simple knot  \(K(p,q,k)\) can be
represented by a doubly pointed diagram of genus one, as described in
section~\ref{SubSec:SimpleKnot}. With respect to this diagram, the
generators of \(\CF(K(p,q,k))\) are just the \(p\) intersection
points \(x_0,x_1,\ldots,x_{p-1}\) between \(\alpha\) and
\(\beta\). All of these intersection points have the same sign. 

\subsection{$\spinc$ structures  and the Alexander grading} In order
to describe
the differential on \(\CF(K)\), we must introduce some more notation.
The alpha and beta curves define a cellulation of \(\Sigma\). The
vertices of this cellulation
 are the intersection points \(\alpha_j \cap \beta_k\),
the one-cells are arcs on the \(\alpha_j\) and \(\beta_k\), and the
two-cells are the components of \(\Sigma - {\alphas} - {
\betas}\). 

Given two generators \({\bf x}\) and \(\bf y\), we can
construct a one-chain \(\eta(\x,\y)\) by going from
points in \({\bf x}\) to points in \({\bf y}\) along the alpha
curves, and then from points in \({\bf y}\) back to points in \({\bf
  x}\) along the beta curves. We can change 
\(\eta(\x,\y)\) by
adding  copies of the \(\alpha_i\)'s and
\(\beta_j\)'s to it, but it has a well-defined image \(\epsilon({\bf
    x}, {\bf y})\) in 
\( H_1(\Sigma)/ \langle \alpha_1,
   \dots, \alpha_g , \beta _1, \ldots, \beta_g \rangle
\cong H_1(Z).\)
This {\it\(\epsilon\)-grading} is additive, in the sense that 
\begin{equation*}
\epsilon({\bf x}, {\bf y}) + \epsilon({\bf y},{\bf z}) = \epsilon({\bf
  x} ,{\bf z}).
\end{equation*}
We define an equivalence relation on the set of generators by setting
\({\bf x} \sim {\bf y}\) if \(\epsilon({\bf x},{\bf y}) = 0\). The
set of equivalence classes is an affine set isomorphic to \(H_1(Z)\). 

If  we fix a basepoint \(q \in \Sigma- \alphas - \betas\), the set of
equivalence classes can naturally be identified with the set of
\(\spinc\) structures on \(Z\). We write \(\spi_q(\bfx)\) to denote the
\(\spinc\) structure determined by the pair \((q, \bfx)\).
 Varying \(q\)
 changes  \(\spi_q(\bfx)\)  according to the formula
\begin{equation}
\label{Eq:Spinc}
\spi_{q_1}(\bfx) - \spi_{q_2}(\bfx) = [K_{q_1,q_2}],
\end{equation}
where \(K_{q_1,q_2}\) is the oriented knot determined by the pair of
basepoints \((q_1,q_2)\). 

In the presence of a knot, we can define an enhancement of the
\(\epsilon\)-grading known as the {\it Alexander grading}. To do this,
we consider the same one-chain \(\eta({\bf x}, {\bf y})\), but in the
Heegaard diagram for the knot complement. The image of 
\(\eta({\bf x}, {\bf y})\) defines a well defined element
\begin{equation*}
A({\bf x},\y) \in H_1(\Sigma')/\langle \alpha_1,\ldots,
\alpha_g,\alpha_{g+1},\beta_1,\ldots, \beta_g\rangle \cong H_1(Z-K). 
\end{equation*}
Like the \(\epsilon\)--grading, the Alexander grading is an additive
function. It reduces to the \(\epsilon\)--grading under the
homomorphism \(H_1(Z-K) \to H_1(Z)\). If  \(K\) is a
primitive knot of order \(p\) (so  \(H_1(Z-K) \cong \Z\)), this means that 
two generators
belong to the same \(\spinc\) structure if and only if their Alexander
gradings are congruent modulo \(p\).

\vskip0.05in
\noindent {\bf Example:} Consider the diagram of \(K=K(5,1,2)\) 
in Figure~\ref{Fig:KnotComplement}. A suitable one-chain
\(\eta(x_3,x_1)\) is shown in bold in the figure. By inspection, we
see that
\(\epsilon(x_3,x_1) = 2a\), where \(a\) is the class of the vertical loop at the
left-hand side of the figure. More generally, we have 
\(\epsilon(x_i,x_j) = (i-j)a\). Since \(a\) generates
\(H_1(L(5,1))\), the generators \(x_0,\dots,x_4\) all belong to
different \(\spinc\) structures. The same argument
shows that if \(K= K(p,q,k)\), the generators \(x_0,\ldots,x_{p-1}\)
all represent different \(\spinc\) structures. 

To compute the Alexander grading \(A(x_3,x_1)\), 
we consider the image of the same loop, but in the group 
\(H_1(Z-K) \cong H_1(\Sigma')/\langle\alpha_1,\alpha_2,\beta
\rangle\). The quotient 
\(H_1(\Sigma')/\langle \alpha_1, \alpha_2 \rangle\) is generated by
\(a\) and \(m\), and the image of \(\beta\) in this quotient is
\(5a+2m\). Thus \(H_1(Z-K)\) is generated by an element \(x\) with
\(a = 2x\) and \(m=-5x\). \(\eta(x_3,x_1)\) is homologous to \(2a +
2m\), so \(A(x_3,x_1) = -6x\).

\subsection{Domains}
If \(\x\) and \(\y\) are two generators, we define 
 \(\pi_2({\bf x}, {\bf y})\) (the set of {\it domains} from \(\x\) to
 \(\y\))
to be the set of two-chains
\(\phi\) with the property that \( \partial \phi = \eta(\x,\y)\) for some one-chain \(\eta(\x,\y)\) joining \(\x\) and \(\y\). Thus \(\pi_2({\bf x},{\bf
  y})\) is empty unless \(\x\) and \(\y\) belong to the same
\(\spinc\) structure.
 In the latter case, assuming that \(Z\) is a rational homology
sphere, there is a unique choice of
\( \eta \) which bounds a two-chain
in \(\Sigma\). Thus \(\pi_2({\bf x}, {\bf y})\) is an affine
copy of \(\Z\), where the action of \(\Z\) is given by adding
multiples of \(\Sigma\). 

If \(\phi \in \pi_2({\bf x}, {\bf y})\) and \(q \in \Sigma - \alphas -
\betas\), then \(\phi\) has a well-defined multiplicity
\(n_{q} (\phi)\) at \(q\). When \(\x\) and \(\y\) belong to the same
\(\spinc\) structure, their Alexander gradings are related by the
following formula: 
\begin{equation}
\label{Eq:Anwnz}
A(\x,\y) = p (n_{ w}(\phi) - n_{ z}(\phi))
\end{equation}
for any \(\phi \in \pi_2(\x,\y)\). 

 
\subsection{The Floer chain complex} 
We are now in a position to describe the differential on
\(\CF(K)\). It takes the following form:

\begin{equation*}
d{\bf x} = \sum_{{\bf y} \in \T_\alpha \cap \T_\beta} \quad  \sum_{{\{\phi \in
  \pi_2({\bf x},{\bf y})  \ts \vert \ts n_{z}(\phi)=0\}}} M(\phi) {\bf y}.  
\end{equation*}

The function \(M(\phi)\) is defined by counting certain
pseudo-holomorphic maps associated to the domain \(\phi\). For a
precise formulation of this count in two different contexts, 
see \cite{OS1}, \cite{Lipshitz}; the main thing that we will need to
know about it is that \(M(\phi) = 0\)
unless \(n_q(\phi) \geq 0 \) for every \(q \in \Sigma - \alphas -
\betas\). (Such a \( \phi\) is called a {\it positive domain}.) 

Regarding the form of the differential, note that the
inner sum is empty unless \({\bf x}\) and \({\bf y}\) belong to the
same \(\spinc\) structure. In this case, there is a unique element 
\(\phi_0(\x,\y) \in \pi_2(\x,\y)\) with \(n_{z}(\phi_0) = 0 \),
so the formula for the differential can be rewritten as
\begin{equation*}
d \x = \sum _{\y \sim \x} M(\phi_0(\x,\y)) \y.
\end{equation*} 
In particular, we can decompose \(\CF(K)\)
into a direct sum over \(\spinc\) structures:
\begin{equation*}
\CF(K) \cong \bigoplus_{\spi \in \Spinc(Z)} \CF(K,\s).
\end{equation*}

\vskip0.05in
\noindent {\bf Example:} If we represent
 \(K=K(p,q,k)\) by a genus one Heegaard diagram as in
 Figure~\ref{Fig:KnotComplement}, then \(\CF(K,\spi) \cong
\Z\) for each \(\spi \in \Spinc(L(p,q))\). Since each generator
belongs to a different \(\spinc\) structure, there are no
differentials in the complex \(\CF(K)\). 
\subsection{The knot filtration}
Up to this point, we have not made much use of the knot \(K\). 
Indeed, the 
homology of the complex \(\CF(K)\) is just the ordinary Heegaard Floer
homology \(\hfhat(Z)\) as defined in \cite{OS1}. To put \(K\) into the
picture, we observe that if \(M(\phi_0(\x,\y))\neq 0\), then 
\(n_w(\phi_0(\x,\y)) \geq 0\). From  equation \ref{Eq:Anwnz}, it
follows that \(A(\x,\y) \geq 0\) as well. For ease of notation, let us
pass  (somewhat arbitrarily)
from an affine \(H_1(Z-K)\) grading to an actual \(H_1(Z-K)\) grading
by fixing some generator \({\bf x_0}\) and setting \(A(\x) =
A(\x,{\bf x_0})\). (In the next section, we will see that there is a
canonical way to do this.) 
Then the formula for the differential becomes 
\begin{equation*}
d \x = \sum _{\{\y \ts | \ts A(\y) \leq A(\x)\} } M(\phi_0(\x,\y)) \y.
\end{equation*}
In other words, the Alexander grading defines a
filtration on \(\CF(K)\). The associated graded complex  \(\CF(K,j)\)
is generated by those \(\bfx\) with \(A(\bfx)=j\). Its homology is
denoted by \(\HFK(K,j)\) or (if we sum over all \(j \in \Z\)) by
\(\HFK(K)\), and is called the knot Floer homology. When we need it,
the \(\Z/2\) homological grading is indicated by a subscript:
\(\HFK_i(K,j)\).

\subsection{Fox Calculus and the Alexander polynomial} 
The Fox calculus \cite{Fox1}, \cite{CrowFox} provides a
streamlined method for computing the Alexander grading. We briefly
sketch this relationship here; for more details, see chapter 2 of
\cite{thesis}. 

We start with the Heegaard diagram
\((\Sigma',\alphas',\betas)\) for \(Z-K\) described in
section~\ref{SubSec:Diagrams}. Any such diagram gives rise to a
presentation of \(\pi_1(Z-K)\) as follows. First, we choose
orientations for the alpha and beta curves. 
We associate a generator \(a_i\) to each
 \(\alpha_i\), and a relation \(w_j\) to
each  \(\beta_j\), according to the following rule.
 Starting at an arbitrary point of \(\beta_j\) and with the
empty word \(w\), we transverse the curve, recording each intersection
with an alpha curve (say \(\alpha_k\)) by appending \(a_k^{\pm1}\) to
\(w\), where the sign is determined by the sign of the intersection
between \(\alpha_k\) and \(\beta_j\). 

Let \(|\cdot | : \pi_1(Z-K) \to H_1(Z-K)\) denote the abelianization
map. For any word \(w\) in the \(a_i\), we define the free differential 
\(d_{a_i}w\) to be an element of the group ring \(\Z[H_1(Z-K)]  \)
determined by the following rules:
\begin{align*}
 d_{a_i}a_j & = \delta_{ij} \\
 d_{a_i}(ab) & = d_{a_i}a + t^{|a|}d_{a_i}b \\
 d_{a_i}a^{-1} & = -t^{|a|}d_{a_i}a.
\end{align*}
(In fact, the last rule is a consequence of the preceding two.)

Before we combine terms, the expression \(d_{a_i}w_j\) contains one
monomial for each point in \(\alpha_i \cap \beta_j\). If we formally
expand the expression \(\det (d_{a_i}w_j)_{1\leq i,j\leq g}\),
again without combining terms, we obtain a polynomial with one term
for each generator of the complex \(\CF(K)\).This polynomial encodes
the Alexander grading, in the sense that if  \(\x\) and \(\y\)
correspond to monomials \(\pm t^x\) and \(\pm t^y\),
then \(A(\x,\y) = x-y\). It also encodes the \(\Z/2\) homological
grading: if two generators have the same \(\Z/2\) grading, the
corresponding monomials have the same sign, and if the gradings are
opposite, their monomials have opposite signs. 

Combining terms in this expression corresponds to the operation of
taking the  {\it graded Euler
characteristic}. More precisely, we have 

\begin{align*}
\chi(\HFK(K)) 
& = \sum _{i,j} (-1)^i t^j \dim \HFK_i(K,j) \\ 
& = \sum_{\x \in \mathbb{T}_\alpha \cap \mathbb{T}_\beta} (-1)^{\gr(\x)} t^{A(\x)} \\
& =  \det (d_{a_i} w_j)_{1\leq i,j\leq g}.
\end{align*}

The matrix \(A= (d_{a_i}w_j)
\), where \(1\leq i\leq g+1\) and \(1 \leq j \leq g\), is known as the 
{\it Alexander matrix}. 
The Alexander polynomial  \(\Delta_K(t)\) is defined to be the
 \(\gcd\) of its  \(g\times g\) minors. 

\begin{prop}
\label{Prop:Chi}
Let \(K\subset Z\) be a primitive knot of order \(p\). 
Then 
\begin{equation*}
\chi(\HFK(K)) \sim \Delta_K(t) \cdot \frac{t^p-1}{t-1}. 
\end{equation*}
\end{prop} 

\noindent Here we write \(f \sim g\) to indicate \(f = \pm t^k g.\)
(This ambiguity arises because
 the \(\gcd\) is only well defined up to multiplication by \(\pm t^k\).)

\begin{proof}
In light of our comments above, this amounts to showing that 
\begin{equation*}
\det(d_{a_i}w_j)_{1\leq i,j\leq g} \sim \Delta_K(t) \cdot \frac{t^p-1}{t-1}.
\end{equation*}
That this is true was certainly known to Fox ({\it c.f.} item 6.3 of
\cite{Fox2}), who actually attributes it to Alexander
\cite{Alexander}. Since the proof is perhaps less well-known to a
modern audience, we sketch it here. 
The rows of
the Alexander  matrix form \(g+1\) vectors \({\bf v}_1,{\bf v}_2,\ldots {\bf
  v}_{g+1}\) in a \(g\)-dimensional space, so 
there must be a linear relation between them. This relation is given
by Fox's {\it fundamental formula}, which implies that for any word \(w\)
\begin{equation*}
t^{|w|}-1 = \sum_{i=1}^{g+1} d_{a_i} w \cdot (t^{|a_i|}-1).
\end{equation*}
(In fact, an analogous relation holds in the group ring of the
free group as well.) When \(w\)=\(w_j\) is a relation in
\(\pi_1(Z-K)\), the left-hand side of this equation is \(0\). It
follows that the \({\bf v}_i\) satisfy the equation
\begin{equation*}
\sum_{i=1}^{g+1} (t^{|a_i|}-1) {\bf v}_i = 0.
\end{equation*}

Let \(\Delta_i\) be the determinant of the \(g\times g\) matrix
obtained by deleting the \(i\)-th row of \(A\). By solving for \({\bf
  v}_j\) in the above equation and substituting it into the expression
for \(\Delta_i\), we find that 
\begin{equation*}
\frac{\Delta_i}{\Delta_j} = \pm \frac{t^{|a_i|}-1}{t^{|a_j|}-1}.
\end{equation*}
 Since the \(a_i\) generate \(\pi_1(Z-K)\), their abelianizations
 generate \(H_1(Z-K) \cong \Z\). In other words,
\(\gcd(|a_i|) = 1\), which implies that
 \(\gcd (t^{|a_i|}-1) = t-1\). Knowing this, it
is not difficult to see that 
\begin{equation*}
\Delta_i \sim \Delta_K(t) \cdot \frac{t^{|a_i|}-1}{t-1}. 
\end{equation*}
The desired formula is a special case, since \(|a_{g+1}|=|m|= \pm p\). 
\end{proof}

It is a well-known fact that the Alexander polynomial \(\Delta_K(t)\) can be
normalized so that \(\Delta_K(t^{-1}) = \Delta_K(t)\),
\(\Delta_K(1)=1\). We use this normalization to fix particular values
for the Alexander and homological gradings on \(\CF(K)\), by requiring
that 
\begin{equation*}
\chi(\HFK(K)) = \DBar(K) = \Delta_K(t) \cdot
\frac{t^{p/2}-t^{-p/2}}{t^{1/2}-t^{-1/2}}.
\end{equation*}
is a symmetric Laurent polynomial with \(\DBar_K(1) = p\). 

\vskip0.05in
\noindent {\bf Example:} Let \(K = K(5,1,2)\). Referring to the
diagram in Figure~\ref{Fig:KnotComplement}, we let \(a\) be the
generator of \(\pi_1(Z-K)\) corresponding to \(\alpha_1\), and \(m\)
be the generator corresponding to \(\alpha_2\). If we traverse
\(\beta\) starting just
below the point \(x_1\), we find that the corresponding relator 
 is \(w=amama^3\). The abelianization map \( |\cdot |:
\pi_1(Z-K) \to H_1(Z-K)\) satisfies \(|a|= 2\), \(|m|={-5}\), so
\begin{align*}
d_a w & = 1 + t^{|am|} + t^{|amam|} + t^{|amama|} + t^{|amama^2|} \\
& = 1 + t^{-3}+t^{-6} + t^{-4} + t^{-2}.
\end{align*}
Thus 
\begin{equation*}
\DBar(K) = t^{-3}+t^{-1}+1 + t + t^3
\end{equation*}
and  the Alexander gradings of
\(x_1,x_2,x_3,x_4\), and \(x_0\) are \(3,0,-3,-1,\) and \(1\)
respectively. The Alexander polynomial of \(K\) is 
\begin{equation*}
\Delta_K(t) \sim \frac{(t-1)d_a(w)}{t^5-1} \sim t^{-1}-1+t 
\end{equation*}
This is recognizable as the Alexander polynomial of the trefoil knot
in \(S^3\). In fact,  \(Z = L(5,1)\) is realized by \(-5\) surgery on the
trefoil, and \(K\) is the dual knot in \(Z\).

\subsection{Reversing orientation} We now consider the effect of exchanging
the roles of \(z\) and \(w\) in the definition of \(\CF(K)\), so that
instead of considering domains with \(n_z(\phi) = 0 \), we use
domains with \(n_w(\phi) = 0 \). Switching the basepoints has the
effect of reversing the orientation on \(K\), so we denote
the resulting complex by \(\CF(-K)\). This complex has the same
generators as \(\CF(K)\), but the differentials are different. 
From equation~\eqref{Eq:Anwnz},
we see that the Alexander grading defines an {\it increasing} filtration on
\(\CF(-K)\), {\it i.e.}
\(d\x\) is a sum of generators \(\y\) with \(A(\y) \geq A(\x)\). 

The \(\epsilon\)-grading on \(\CF(-K)\) remains the the same as on \(\CF(K)\),
but the \(\spinc\) structure determined by an equivalence class will
differ. In order to state the relationship precisely, we denote by
 \(\spi_k\) the \(\spinc\) structure on \(Z\) given by
\(\spi_z(\x)\), where \(\x\) is any generator with \(A(k) \equiv k \
(p)\). Then by combining Lemma~\ref{Lem:LinkingForm} with
 equation~\eqref{Eq:Spinc}, we see that 
\begin{align*}
\spi_w(\x) & = \spi_z(\x) - [K] \\
& = \spi_{A(x)-a},
\end{align*}
where \(a\) is the self-linking number of \(K\). 
In particular, the  summand of \(\CF(-K)\) generated by those \(\x\) with
\(A(\x) \equiv k \ (p)\) has homology equal to \(\hfhat(Z,
\spi_{k-a})\). 

\section{Knots with LHS surgeries}
\label{Sec:LHS}
We now suppose that we are given a knot \(K \subset Z\), where \(Z\) is an
L-space. In this section, we give a precise characterization of when
\(K\) has a surgery which is an L-space homology sphere in terms of
the knot Floer homology of \(K\). The main tool is the mapping cone
theorem of \Ozsvath and \Szabo \cite{OSFrac}, which expresses the
Heegaard Floer homology of surgeries on \(K\) in terms of the homology
of certain complexes derived from \(\CF(K)\) and \(\CF(-K)\). We begin
by recalling  their construction. 

\subsection{The complex \(C_n(K)\)}

The differential in the
complex \(\CF(K)\) can be decomposed as \(d=d_0+d_+\), where 
\begin{align*}
d_0(\x) & = \sum_{A(\y)=A(\x)} \quad  \sum_{\{\phi\in \pi_2(\x,\y)
    \ts \vert \ts 
  n_z(\phi)=n_w(\phi)=0\}} M(\phi)\y \\
d_+(\x) & = \sum_{A(\y)<A(\x)} \quad \sum_{\{\phi\in \pi_2(\x,\y) \ts \vert \ts
  n_z(\phi)=0\}} M(\phi)\y \\
\end{align*}
Similarly, the differential in \(\CF(-K)\) can be decomposed as
\(d_0+d_-\), where 
\begin{equation*}
d_-(\x)  = \sum_{A(\y)>A(\x)}  \quad \sum_{\{\phi\in \pi_2(\x,\y) \ts \vert \ts
  n_z(\phi)=0\}} M(\phi)\y 
\end{equation*}

For each \(n \in \Z\), we let \(C_n(K)\) be the complex 
generated by those \(\bfx \in \mathbb{T}_\alpha \cap \mathbb{T}_\beta
\) for which \(A(\x) \equiv n \ (p)\), and whose differential is given
by the formula 
\begin{equation*}
d_n \x = \begin{cases} 
d_0(\x) + d_+(\x) & \quad \text{if} \ A(\x) < n\\
 d_0(\x) + d_+(\x) + d_-(\x) & \quad \text{if} \ A(\x) = n \\
      d_0(\x)   + d_-(\x) & \quad \text{if} \ A(\x) > n.
\end{cases}
\end{equation*}
When \(n\gg 0\), \(C_n(K) = \CF(K, \spi_n)\), while for \(n \ll 0\) , 
\(C_n(K) = \CF(-K, \spi_{n-a})\). There are natural maps
$\pi_n^+: C_n(K) \to \CF(K,\spi_n)$ and $\pi_n^-: C_n(K) \to \CF(-K,\spi_{n-a}) $
defined by 
\begin{equation*}
\pi_n^+(\x) = \begin{cases} 
\x & \quad \text{if} \ A(\x) \leq n\\
  0 & \quad \text{if} \ A(\x) > n
\end{cases}
\quad \text{and} \quad 
\pi_n^-(\x) = \begin{cases} 
0 & \quad \text{if} \ A(\x) < n \\
  \x & \quad \text{if} \ A(\x) \geq n.
\end{cases}
\end{equation*}
We denote the homology group \(H(C_n(K),d_n)\) by \(A_n\), and let 
\begin{equation*}
\pi_n^+: A_n \to \hfhat(Z,\spi_n) \quad \text{and} \quad 
\pi_n^-: A_n \to \hfhat(Z,\spi_{n-a})
\end{equation*}
be the induced maps.

Geometrically speaking, the group \(A_n\) can be identified with
 \(\hfhat(Z',\spi_n)\), where \(Z'\) is a manifold obtained by doing a
 large integral surgery on \(K\), and \(\spi_n\) is a particular
 \(\spinc\) structure on \(Z'\) ({\it c.f.} section 4 of
 \cite{OSFrac}). The maps \(\pi_n^\pm\) are induced by
 certain \(\spinc\) structures on the surgery cobordism. An easy (but
 useful) consequence of this identification is that each \(A_n\) must
 have rank \(\geq 1\). 
 
\subsection{The mapping cone formula} 
The formula of \cite{OSFrac} expresses the homology of surgeries on
\(K\) in terms of the groups \( A_n\) and their projections
\(\pi_n^\pm\) to \(\hfhat(Z)\). To be precise, 
 recall from  Lemma~\ref{Lem:LinkingForm}
 that the first homology groups of the manifolds \(K_m\) obtained by 
integer surgery on \(K\) are precisely of the form \(\Z/m\), where 
\( m \equiv - a \ (p)\). 

We now fix some 
  \(m \equiv -a \ (p)\). For each \(n \in \Z\), we define \(
B_n = \hfhat(Z,\spi_n) \). (Note that although \(B_{n_1} \cong
B_{n_2}\) whenever \(n_1 \equiv n_2 \ (p)\), we treat them as
different groups.) Then we have maps 
\begin{align*}
\pi_n^+&: A_n \to B_n  \\
\pi_n^- &:A_n \to B_{n+m}.
\end{align*} 
We write 
\begin{equation*}
A = \bigoplus _{n \in \Z} A_n \quad \text{and} \quad B = \bigoplus_{n \in \Z} B_n.
\end{equation*}
and let \(\pi^\pm:A\to B\)  be the maps whose
components are given by \(\pi^\pm_n\). 
 Then we can form the short chain complex 
\begin{equation*}
\begin{CD}
C(K,m) = A @>{\pi^- + \pi^+}>> B .
\end{CD}
\end{equation*}

\begin{thrm} \cite{OSFrac}
\label{Thrm:MappingCone}
 \(\hfhat(K_{m})\) is isomorphic to the homology of the complex \(C(K,m)\).
\end{thrm}
A few remarks are in order. First, we should point out that 
 the homology of the complexes \(\CF(K,\spi_n)\) and \(\CF(-K,
\spi_n)\) both compute \(\hfhat(Z,\spi_n)\), so they are 
canonically isomorphic ({\it c.f.} Theorem 2.1 of
\cite{OS3}). In general, however, there is no easy way to determine this
isomorphism. Thus even though 
the maps \(\pi^+\) and \(\pi^-\) are  determined by the
complex \(\CF(K)\),  the behavior of their sum  can be quite difficult
to calculate. However, we will  only consider the case where
\(\hfhat(Z,\spi_n) \cong \Z\) (which has very few isomorphisms), so
this difficulty will not arise.

Second, note that since \(\pi^+\)
preserves  \(n\), and  \(\pi^-\) raises it by
\(m\), \(C(K,m)\) can be decomposed into a direct sum of \(m\)
complexes. This splitting 
corresponds to the decomposition of \(\hfhat(K_m)\) into \(\spinc\)
structures. 

Finally, observe that for \(n \gg 0\), the map \(\pi_n^-\) is
trivial, and \(\pi_n^+\) is an isomorphism. Similarly, for \(n \ll
0\), \(\pi_n^-\) is an isomorphism, and \(\pi_n^+\) is trivial.  It
follows that the chain complex \(C(K,m)\) (which is infinitely
generated) can be decomposed into an infinite number of  summands of
the form
\begin{equation*}
\begin{CD}
A_n @>\pi_n^+>> B_{n} \quad (n >N_+ ) \quad \text{and} \quad A_n
@>\pi_n^->> B_{n+m}  \quad (n < N_- )
\end{CD}
\end{equation*}
whose homology is trivial, together with a single interesting summand
\(\widetilde{C}(K,m)\) which contains \(A_n\) for \(N_-\leq n \leq
N_+\) and \(B_n\) for \(N_-+m\leq n \leq N_+\).

\subsection{Proof of Theorem~\ref{Thrm:LHSCrit}}

Suppose now that \(Z\) is an L-space. We wish to characterize when
\(K\) has an L-space homology sphere surgery in terms of
\(\HFK(K)\). To do so, we recall a few invariants derived
from the knot Floer homology. 

\begin{defn} The {\em width} of \(\HFK(K)\) is the difference 
\(M_+-M_-\), where \(M_+\) is the maximum value of \(j\) for which
\(\HFK(K,j)\) is nontrivial, and \(M_-\) is the minimum value. 
\end{defn}

The width is related to the genus of \(K\) by the following theorem of
Ni:
\begin{thrm}
\label{Thrm:YiQHS}
\cite{YiQHS}
Suppose \(K\) is a primitive knot in a rational homology sphere \(Z\),
and that \(H_1(\Z) \cong \Z/p\). Then \(g(K) =
(\text{\em{width}} \ \HFK(K)-p+1)/2\). 
\end{thrm}

The other invariant is an obvious generalization of the
Ozsv{\'a}th-Szab{\'o} \(\tau\) invariant defined in \cite{OS10} ({\it
  c.f.}  \cite{HeddenTau}, \cite{thesis}):

\begin{defn}
If \(\spi_k\) is a \(\spinc\) structure on \(Z\), we define
\(\tau(K,\spi_k)\) to be the minimum value of \(n \equiv k \ (p)\) 
for which the map \(\pi^+_n\) is nontrivial.  
\end{defn}

\begin{prop}
\label{Prop:LSCrit}
Suppose \(Z\) is an L-space, and that
 \(K \subset Z\) has a homology sphere surgery \(Y\). Then
\(Y\) is an L-space if and only if one of the following conditions
holds:
\begin{enumerate}
\item  \(\HFK(K) \cong \Z^p\) and \(\text{\em width} \ \HFK(K) <2p\).
\item  \(\HFK(K) \cong \Z^{p+2}\),  \(\text{\em width}\ \HFK(K) =2p\), and
  either \(Y=K_{1}\) and \(\tau(K,\spi_0)>0\) or 
\(Y=K_{-1}\) and \(\tau(K,\spi_0)<0\).
\end{enumerate}
\end{prop}

We will work our way up to the proof through a series of lemmas. 
For the rest of this section, we suppose that \(Z\) is an L-space and
that \(K\subset Z\) admits a homology
sphere surgery. For the moment, we assume that this surgery is \(K_{1}\).


\begin{lem}
\label{Lem:An}
If \(K_{1}\) is an L-space, then 
\(A_n \cong \Z\) for every \(n \in \Z\), and there is at most one value
of \(n\) for which both \(\pi_n^+\) and \(\pi_n^-\) are both trivial. 
\end{lem}

\begin{proof}
Since \(Z\) is an L-space, \(B_n \cong \Z\) for every \(n \in
\Z\). Thus \(A_n \cong B_{n} \cong \Z \) for all \(n> N_+\) and \(A_n
\cong B_{n+1} \cong \Z\) for all \(n <N_-\). For the intermediate values of
\(n\), we consider the complex \(\widetilde{C}(K,1)\). 
The first term in this complex is the direct sum of the \(A_n\) for
\(N_-\leq n \leq N_+\), while the second term is
 the direct sum of the \(B_n\) for \(N_-+1 \leq n \leq N_+\). In
 particular, the first term has one more summand than the second. Now
 each \(A_n\) has rank \(\geq 1\), and each \(B_n\) has rank \(1\). On
 the other hand, we have 
\begin{equation*}
H(\widetilde{C}(K,1)) \cong H(C(K,1)) \cong \hfhat(K_1) \cong \Z
\end{equation*}
since \(K_1\) is an L-space. The only way this can happen is if 
 \(A_n \cong \Z\) for each \(N_-\leq n\leq N_+\). This proves the first claim. For the
second, note that if  both \(\pi_n^+\) and \(\pi_n^-\) are trivial for two
different values of \(n\), then the homology of \(\widetilde{C}(K,1)\) must
have  rank \(\geq 2\). 
\end{proof}

\begin{lem}
\label{Lem:pp2}
If \(K_{1}\) is an L-space, 
either \(\HFK(K) \cong \Z^p\) or \(\HFK(K) \cong \Z^{p+2}\).
\end{lem}
\begin{proof}
Consider the complex \(\CF(K,\spi_k)\) generated by those \({\bf x}\)
with \(A({\bf x}) \equiv k \ (p)\). By the previous lemma, we know
that the associated homology groups \(A_n \) (\(n \equiv k \ (p)\))
are all isomorphic to \(\Z\). This situation was studied by \Ozsvath
and \Szabo in Lemmas 3.1 and 3.2 of \cite{OSLens}. They show that
there is a series of integers \(n_1<n_2<\ldots<n_{2m_{k}-1}\)
such that \(\HFK(K,\spi_k,n_j) \cong \Z\), and that
\(\HFK(K,\spi_k,n)\) is trivial for all other values of
\(n \equiv k \ (p) \). 
Furthermore, if \(y_{j}\) is a  generator of the group in Alexander
grading \(n_j\), then 
\begin{align*}
d_+(y_{2j}) & = \pm y_{2j+1} \quad \quad  d_+(y_{2j+1})  = 0  \\
d_-(y_{2j}) & = \pm y_{2j-1} \quad \quad d_-(y_{2j-1})  = 0.
\end{align*}
From this, it is easy to see that the maps \(\pi_n^{\pm}\) are both
trivial for all  \(n_1 < n<  n_{2m_k-1}\), \(n \equiv k \ (p)\). If
\(m_k>1\), there is at least one such value of \(k\), and if
\(m_k>2\), there are more than one. By
the preceding lemma, we conclude that \(m_k = 1\) (and thus
\(\HFK(K,\spi_k) \cong \Z\)) for all but one value of \(k\), and that 
\(m_k=2\) (so \(\HFK(K,\spi_k) \cong \Z^3\)) for this value, if it
exists. 
\end{proof}

\begin{lem}
\label{Lem:2p}
Suppose  \(\HFK(K) \cong \Z^p\). Then \(K_{1}\) is an L-space if and
only if $$\text{\em{width}} \ \HFK(K)<2p.$$
\end{lem}

\begin{proof}
The argument in the preceding lemma shows that 
for each \(k \in \Z/p\), there is a  unique
\(n_k \equiv k \ts (p)\)
with \(\HFK(K,n_k) \cong \Z\), and that all the other groups
\(\HFK(K,n)\) vanish. From this, we see that \(A_n \cong \Z\) for all
\(n\), and that  \(\pi_n^+\) is an isomorphism for all \(n \geq
n_k\), \(n \equiv k \ (p)\) and vanishes for all \(n <n_k\). 
Similarly,  \(\pi_-^n\) is an isomorphism for all \(n \leq
n_k\),    \(n \equiv k \ (p)\), and vanishes for all \(n>n_k\). 

\begin{figure}
\includegraphics{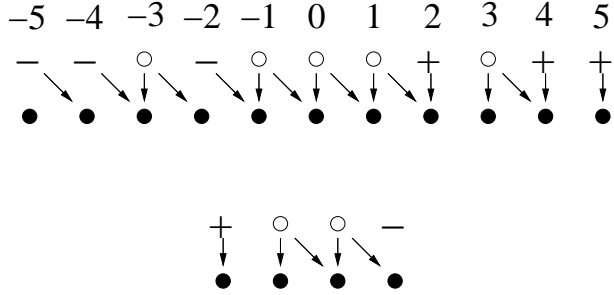}
\caption{\label{Fig:PlusMinusD} The top half of the figure is a diagram
  of the complex \(\widetilde{C}(K,1)\) for \(K=K(5,1,2)\). The
  numbers label the indices on the groups \(A_n\) and \(B_n\), which
  are indicated by \(+\)'s, \(-\)'s and \(o\)'s (for the \(A_n\)), and
  solid dots for the \(B_n\). Summands of type \([-,-]\), \([-,+]\), and
\([+,+]\) are visible. The bottom half of the figure shows a typical
summand of type \([+,-]\).}
\end{figure}

We can represent the chain complex \(C(K,1)\) by a diagram of the
type illustrated in Figure~\ref{Fig:PlusMinusD}. The upper row of the
diagram shows the \(A_n\), while the lower row shows the \(B_n\).
We represent the
group \(A_n\) by a \(+\) if \(\pi_n^+\) is nontrivial but \(\pi_n^- =0\), by 
a \(-\) if  \(\pi_n^-\) is nontrivial but \(\pi_n^+ =0\), and by an
\(o\) if both maps are nontrivial. (Thus \(A_{n_k+ip}\) is represented
by the sign of \(i\).) Each \(B_n\) is represented by a filled circle. Nontrivial maps are indicated by arrows, but
trivial ones are omitted. 

The complex \(C(K,1)\) can be decomposed into summands corresponding
to connected components of the diagram. Each summand corresponds to an
interval \([a,b]\), where \(a\) and \(b\) are labeled with a \(+\) or
\(-\), and all the intervening integers are labeled with an \(o\). The
homology of a summand of type \([+,+]\) or \([-,-]\) is trivial, but
the homology of a summand of type \([-,+]\) is \(\Z\) (supported
in the top row) and the homology of a summand of type \([+,-]\) is
\(\Z\) (supported in the bottom row). Thus there is a unique summand
with nontrivial homology if and only if  the \(-\)'s in the diagram
appear to the left of all the \(+\)'s. 

Suppose \(A_m\) is labeled with
a \(+\), and \(A_n\) is labeled with a \(-\) for some \(m<n\). Then
\(\HFK(K,m-ip) \cong \Z\) for some \(i>0\), and \(\HFK(K,n+ip) \cong
\Z\) for some \(i>0\), so \(\text{width} \ \HFK(K) > 2p\). Conversely,
if \(\HFK(K,m) \cong \HFK(K,n) \cong \Z\) with \(n-m > 2p\), then
\(A_{m+p}\) is labeled with a \(+\) and lies to the left of
\(A_{n-p}\), which is labeled with a \(-\). This proves the claim.
\end{proof}

\begin{lem}
\label{Lem:Tref}
Suppose  \(\HFK(K) \cong \Z^{p+2}\). Then \(K_{1}\) is an L-space if and
only if 
\begin{equation*}
\text{{\em width}} \ \HFK(K) =2p \quad \text{ and} \quad \tau(K,\spi_0)>0.
\end{equation*} 
\end{lem}

\begin{proof}
In this case, there is a unique \(\spinc\) structure  with
\(\HFK(K,\spi) \cong \Z^3\). The symmetry of \(\Delta_K(t)\) implies
that this is necessarily \(\spi_0\). If \(K_{1}\) is an L-space, the
 argument used in the proof of
Lemma~\ref{Lem:pp2} shows that the three \(\Z\) summands are 
in Alexander gradings \(-p,0\), and \(p\), and that \(\tau(K,\spi_0) =
p \). Conversely, if  \(\text{{width}} \ \HFK(K) =2p\), then 
 \(\HFK(K,\spi_0)\) must be 
 supported in Alexander gradings \(-p,0\), and \(p\). It is now
 easy to see that \(\pi_{ip}^+ \) is an isomorphism for
\(i>0\) and vanishes for \(i\leq 0\), and that  \(\pi_{ip}^- \) is an
isomorphism for \(i<0\) and vanishes for \(i \geq 0\). All the
other \(\spinc\) structures behave exactly as they did in the
proof of Lemma~\ref{Lem:2p}.

We represent the chain complex \(C(K,1)\) by the same sort of
diagram we used in the proof of Lemma~\ref{Lem:2p},
 labeling \(A_0\) with a \(*\), and 
each \(A_n\) (\(n \neq 0\)) by either a
\(-\), a \(+\), or an \(o\). As before, \(C(K,1)\)
decomposes into summands corresponding to intervals \([a,b]\) all of
whose interior points are labeled by an \(o\); however, there are now
some additional possibilities. First, \(A_0\) itself is always a
summand, with homology \(\Z\). Second, intervals of the form
\([-,*]\) and \([*,+]\) correspond to summands with trivial homology,
while those of the form \([+,*]\), and \([*,-]\) have homology
\(\Z\). From this, it is easy to see that \(K_{1}\) 
is an L-space if and only if the
groups labeled with a \(+\) all have \(n>0\), and those labeled with a
\(-\) all have \(n<0\). This happens if and
only if all the groups labeled with an \(o\) have \(-p<n<p\), which is
equivalent to the statement that \(\text{width} \ \HFK(K) = 2p\). 
\end{proof}

\begin{proof}[Proof of Proposition~\ref{Prop:LSCrit}]
If \(Y=K_{1}\), this is an immediate consequence of
Lemmas~\ref{Lem:An}--\ref{Lem:Tref}. If \(Y=K_{-1}\), we
consider the mirror knot \(\overline{K} \subset \overline{Z}\), for
which \(\overline{K}_{1} = \overline{K_{-1}} =
\overline{Y}\). \(\overline{Y}\) is an L-space if and only if \(Y\)
is, so the claim follows from the previous case, together with the 
identities \(\HFK(K,j) \cong \HFK(\overline{K},j)\) and
\(\tau(\overline{K}, \spi_0) =-\tau (K,\spi_0)\). (These are
well-known  when \(K\) is null-homologous, and their proof
carries over to our situation without change.)
\end{proof}

\begin{proof}[Proof of Theorem~\ref{Thrm:LHSCrit}]
By Theorem~\ref{Thrm:YiQHS}, \(\text{width} \ \HFK(K) < 2p\) if and
only if
\(g(K)< (p+1)/2\). In light of the proposition, it suffices to show
that \(\HFK(K) \cong \Z^p\) whenever the width of \(\HFK(K)\) is less
than \(2p\). Suppose that \(\text{width} \ \HFK(K) < 2p\), and that
there is some \(k\) for which \(\HFK(K,\spi_k) \not \cong \Z\).
 Then we can find a
prime \(q\) so that \(\dim _{\Z/q} \HFK(K,\spi_k;\Z/q) > 1\). By
hypothesis, \(\HFK(K,\spi_k; \Z/q)\) is supported in at most two 
Alexander gradings --- call them \(k\) and \(k-p\).

As described in section 3 of \cite{thesis}, we can find a {\it
  reduced} complex \((C',d_+')\) which is filtered chain homotopy equivalent to 
the complex \(\CF(K,\spi_k;\Z/q)\), and whose underlying group is
isomorphic to  
the direct sum of \(\HFK(K,j; \Z/q)\) for \(j\equiv k \ (p)\). 
 On the other hand, the fact that \(Z\) is an L-space, combined with
 the universal coefficient theorem tells us that
\begin{equation*}
H(\CF(K,\spi_k;\Z/q)) \cong \hfhat(Z,\spi_k;\Z/q) \cong \Z/q
\end{equation*}
It follows that 
$\dim_{\Z/q}  \HFK(K,k; \Z/q)  = \dim_{\Z/q} \HFK(K,k-p; \Z/q) \pm 1.$
Without loss of generality,
 let us assume that \(\HFK(K,k; \Z/q)\) is larger. Then the induced
 differential \(d_+': 
\HFK(K,k; \Z/q) \to \HFK(K,k-p; \Z/q) \) is surjective. 

Similarly, there is reduced complex \((C',d_-')\) which is filtered
chain homotopy equivalent to \(\CF(-K,\spi_{k-a};\Z/q)\), and the induced
differential \(d_-':\HFK(K,k-p; \Z/q) \to \HFK(K,k; \Z/q) \) must be
injective. Thus there is some 
is some \(x \in \HFK(K,k; \Z/q)\) for which \(d_-'d_+'x \neq 0\). Now
 \(d_-'x\)  vanishes for grading reasons, so \(d_-'d_+'x + d_+'d_-'x \neq
0\). 

On the other hand,  consider the
  bifiltered complex
  \(CFK^\infty(K)\),  defined in \cite{OS7},
  \cite{thesis}. In the corresponding reduced complex
\((C_\infty',d_\infty')\), 
 \(d_-'\) and \(d_+'\)  are the components of \(d_\infty'\) which
 lower the bifiltration by \((1,0)\) and \((0,1)\) respectively. 
Thus  \(d_-'d_+' + d_+'d_-'\) is the component of 
\((d_\infty')^2\) which lowers the filtration by
\((1,1)\). It follows that   
 \(d_-'d_+' + d_+'d_-' = 0\), so we have reached  a contradiction. 
\end{proof}

%

\section{The Fox -Brody theorem and applications}
\label{Sec:FoxBrody}

In this section, we prove Theorems~\ref{Thrm:Simple} and
\ref{Thrm:MinGenus}. The main ingredient is an old theorem of Fox
and Brody. To state it, recall that if \(K \subset M\) is a knot in a
three-manifold, the Alexander polynomial \(\Delta_K\) is most naturally
viewed as an element of the group ring  \(\Z[H_1(M-K)]\). 

\begin{thrm}[The Fox-Brody Theorem]
\cite{Brody} Suppose that \(K\subset M\) is a knot in a
three-manifold and that \(H_1(M-K)\) is torsion-free. If 
\(
i_*:\Z[H_1(M-K)] \to \Z[H_1(M)] 
\)
is the map induced by inclusion, then the ideal generated by
\(i_*(\Delta_K)\) depends only on the class of \([K]\)  in \(H_1(M)\). 
\end{thrm}

In other words, 
 if \(K_1\) and \(K_2\) are knots representing
the same homology class in \(M\), then \(i_*(\Delta_{K_1}) =
i_*(\Delta_{K_2})\) up to multiplication by units in the group ring 
\(\Z[H_1(M)]\). When  \(H_1(M) \cong
\Z/p\), this amounts to
saying that  \(\Delta_{K_1}(t) = \pm t^k \Delta_{K_2}(t)\) in the ring
\(\Z[\Z/p] \cong \Z[t]/(t^p-1)\). This uncertainty can presumably be
eliminated using the Turaev torsion ({\it c.f.} Theorem VII.1.4 in
\cite{TuraevTorsion}, which unfortunately does not cover our
situation.) However, in this case  we can achieve the
same result by elementary means. 

\begin{lem}
Suppose that \(H_1(Z) \cong \Z/p\), and that \(K_1\) and \(K_2\) are primitive
knots representing the same homology class in \(Z\). If we normalize 
\(\Delta_{K_i}(t)\) \((i=1,2)\) so that \(\Delta_{K_i}(1) = 1\) and  \(\Delta_{K_i}(t^{-1}) =
\Delta_{K_i}(t)\), then  \(i_*(\Delta_{K_1}) = i_*( \Delta_{K_2})\).
\end{lem} 

\begin{proof}
As noted above, the Fox-Brody theorem implies that 
 \(i_*(\Delta_{K_1}(t)) = \pm t^k
i_*(\Delta_{K_2}(t))\) in \(\Z[\Z/p]\). The requirement
that \(\Delta_{K_1}(1) = \Delta_{K_2}(1) = 1\) ensures that the sign
is positive. To see that \(k=0\), we
view the polynomial \(i_*(\Delta_{K_1}(t))\) as assigning a number to each
\(p\)th root of unity in the complex plane. The symmetry of  \(\Delta_{K_1}\)
says that the resulting diagram is invariant under reflection across
the real axis, while the symmetry of \(\Delta_{K_2}\)
implies that the diagram is also invariant under reflection about some
other axis. If the two axes differ, then
the diagram is invariant under the composition of the two
reflections, which is a nontrivial rotation. If the order of this
rotation is \(m\), then \(\Delta_{K_1}(1)\) must be divisible by
\(m\). This contradicts the fact that \(\Delta_{K_1}(1)=1\), so the
two axes are  the same. If \(p\) is odd, it follows
that \(k \equiv 0 \ (p)\), while if \(p\) is even, either \(k \equiv 0\) or
\(k \equiv p/2
\ (p)\). To eliminate the second possibility, note that since
\(\Delta_K(1) = 1\), the coefficient of \(t^0\) in \(\Delta_{K_1}(t)\)
  must be odd. This implies that  the coefficient of \(t^0\)
  in \(i_*(\Delta_{K_1}(t))\) is odd, while the coefficient of \(t^{p/2}\)
  is  even. But if \(k \equiv p/2 \ (p)\), the same argument applied to
  \(\Delta_{K_2}(t)\) shows that the coefficient of \(t^0\)
  in \(i_*(\Delta_{K_1}(t))\) is even, while the coefficient of \(t^{p/2}\)
  is  odd.
\end{proof}

If \(K\subset Z \) is a primitive knot of order \(p\), recall from 
Proposition~\ref{Prop:Chi} that 
\begin{equation*}
\DBar(K) = \Delta_K(t) \cdot \frac{t^{p/2}-t^{-p/2}}{t^{1/2}-t^{-1/2}}.
\end{equation*}
 is the graded Euler characteristic of \(\HFK(K)\). 

\begin{cor}
\label{Cor:ChiDiv}
Suppose that \(H_1(Z) \cong \Z/p\), and that \(K_1,K_2 \subset Z\)
are two primitive knots in the same homology class. Then 
$\DBar(K_1)-\DBar(K_2)$
is divisible by \((t^p-1)^2\). 
\end{cor}

\begin{proof}
We must show  that
\(\delta =\Delta_{K_1}(t) - \Delta_{K_2}(t)\) is divisible by
\((t-1)(t^p-1)\). The Fox-Brody theorem tells us that
\((t^p-1)|\delta\). Thus we need only show \((t-1)^2|\delta\). By the
symmetry of \(\Delta_{K_1}\) and \(\Delta_{K_2}\), 
, we know that \(\delta(t^{-1})=\delta(t)\). Suppose
\(a \neq 0,\pm 1\) is a root of \(\delta\). Then \(a^{-1}\) is also a
root, and we can consider the polynomial \(\delta_1 =
\delta/(t-a)(t^{-1}-a)\), which is also symmetric. Iterating, we
eventually arrive at some  \(\delta_n\) which has no roots other than 
\(\pm 1\), and thus is of the form
$\delta_n = t^k(t+1)^a(t-1)^b.$
Substituting \(t=t^{-1}\) and equating, we see that \(b\) must be
even. Since we already know that \((t-1)|\delta\), this proves the claim.
\end{proof}

\begin{proof}[Proof of Theorem~\ref{Thrm:Simple}]
Let \(K'=K(p,q,k)\) be a primitive simple knot in \(L(p,q)\) (so
that \((k,p)=1\)). Then \(\DBar(K')\) has the following property (*):
for each \(i \in \Z/p\), there is a unique \(n \equiv i \ts (p)\) such
that the coefficient of \(t^n\) in \(\DBar(K')\) is
nonvanishing. It is easy to see that
 there is no other polynomial  congruent to 
\(\DBar(K')\) modulo \((t^p-1)^2\) which has this property. 

Suppose  that \(K \subset L(p,q)\) is another knot
representing the same homology class as \(K'\), and that \(K\)
admits an LHS surgery. Then by
Proposition~\ref{Prop:LSCrit},  \(\HFK(K)\) is isomorphic to either
\(\Z^p\) or \(\Z^{p+2}\). In the first case, \(\DBar(K')\) has
property (*), so by Corollary~\ref{Cor:ChiDiv}, we must have
\(\DBar(K) = \DBar(K')\). It follows that \(\HFK(K) \cong
\HFK(K')\), and thus that \(g(K)=g(K')\). Since \(K\) and \(K'\) are
in the same homology class, \(K'\) has a homology sphere surgery, and
 by Proposition~\ref{Prop:LSCrit}, this homology sphere is an
 L-space. 

Next, suppose \(\HFK(K) \cong \Z^{p+2}\). From  the
proof of Lemma~\ref{Lem:Tref}, we know that for \(k\neq 0\),
\(\chi(\HFK(K), \spi_k) = t^i\) for some \(-p<i<p\), and that 
\(\chi(\HFK(K), \spi_0) = t^{-p}-1+t^p\). In particular,
\(f(t) = \DBar(K) - t^{-p}(t^p-1)^2\) has property (*). It
follows that \(f(t)=\DBar(K')\) and thus  \(\text{width} \ts  \HFK(K') <
2p\). Applying  Proposition~\ref{Prop:LSCrit} once more,  we see that 
 \(K'\) has an LHS surgery. 
\end{proof}

\begin{proof}[Proof of Theorem~\ref{Thrm:MinGenus}]
By Corollary~\ref{Cor:ChiDiv}, we can write
\begin{equation*}
\DBar(K') = \DBar(K) + f(t)(t^{-p}-2+t^p),
\end{equation*}
where \(f(t)\) is some symmetric Laurent polynomial. 
By hypothesis, \(g(K)<(p+1)/2\), so the degree of \(\DBar(K)\) is less
than \(p\). It follows that either degree \(\DBar(K') \geq p\), so 
\(g(K') \geq (p+1)/2\), or \(f(t) = 0\), so
\(\DBar(K')=\DBar(K)\). In the latter case, we have 
\begin{align*}
g(K) & \geq \frac{1}{2}(\text{degree} \ \DBar(K')-p+1) \\
& = \frac{1}{2}(\text{degree} \ \DBar(K)-p+1) \\
& = g(K').
\end{align*} 
To see that the last equality holds, observe from the proof of
Theorem~\ref{Thrm:LHSCrit} that  \(\HFK(K)\) must be isomorphic to \(
\Z^p\), so \(\HFK(K)\) is determined by \(\DBar(K)\).   
\end{proof}

\subsection{Knots with width \({\bf = 2p}\)}

In this section, we consider knots \(K \subset L(p,q)\)  which have 
\(\text{width} \ \HFK (K) = 2p\) and admit an integer LHS surgery,
which we assume for the moment is \(K_1\). 
 By Theorem~\ref{Thrm:Simple}, each such
\(K\) is in the same homology class as a simple knot \(K'\) with
\(\text{width} \ \HFK(K') < 2p\). 
  \(K'\)  admits a \(\Z\)HS surgery \(K_1'\), which is an L-space by
Theorem~\ref{Thrm:LHSCrit}. Although
\(\hfhat(K_1) \cong \hfhat(K_1')\), the two spaces can be
distinguished by their \(d\)-invariants \cite{OS4}.

\begin{prop}
\label{Prop:dInvt}
With hypotheses as above, \(d(K_1) = d(K_1') - 2\). 
\end{prop}

\begin{proof}
The \(d\) invariant of \(K_1\) is  the absolute grading of the generator of \(\hfhat(K_1) \cong \Z\). 
To compare \(d(K_1)\) with \(d(K_1')\), we return to the mapping cone
theorem of \Ozsvath and Szab{\'o}. In addition to computing the Floer
homology of surgeries on \(K\), the mapping cone can be used to
determine the maps induced by the corresponding surgery cobordism. 
(This is not explicitly stated in \cite{OSFrac}, but the analogous
result for null-homologous knots may be found in \cite{OSInt}, and
the proof carries through without change.)
The precise statement is as follows:
for each \(n\in \Z\), there is an inclusion of chain complexes \(i_n:
B_n \to C(K,1)\). If  \(W:L(p,q) \to K_1\) is the surgery
cobordism, the set \(\Spinc (W)\) may be identified with \(\Z\) in
such a way that the 
induced map 
\begin{equation*}
\widehat{F}_{W,\spi_n} : \hfhat(L(p,q), \spi_n) \to \hfhat(K_1)
\end{equation*}
is equal (as a relatively graded map) to  the map induced by \(i_n\). 

Let \(x\) be a generator of \(\hfhat(L(p,q),\spi_0)\). Although
\((i_0)_*(x)\)  is trivial in homology, it still makes sense to talk
about its homological grading. Inspecting \(C(K,1)\), we see that  the
grading of the generator of \(\hfhat(K_1)\) is one less than that of
\((i_0)_*(x)\). On the other hand, a similar computation with
\(C(K',1)\) shows that the grading of the generator of
\(\hfhat(K_1')\) is one {\it more} than that of  \((i_0')_*(x)\).

 To complete the proof, we observe that the surgery cobordism \(W':L(p,q) \to K_1'\) has exactly the same homological properties as \(W\), and that \(c_1(\spi_0)^2 = c_1(\spi_0')^2\). (To see this, consider the conjugation symmetry, which acts as a reflection on the affine \(\Z\) graded sets \(\Spinc(W)\) and \(\Spinc(W')\). Both \(\spi_0\) and \(\spi_0'\) are \(\spinc\) structures nearest to the center of the reflection.)
Thus  \(\widehat{F}_{W,\spi_0} (x)\) and \(\widehat{F}_{W',\spi_0}(x) \) have the same absolute grading. 
\end{proof}

\begin{cor}
If Conjecture~\ref{Conj:Num} is true, then no 
\(K\subset L(p,q)\) with \(g(K)=(p+1)/2\) admits an \(S^3\) surgery.
\end{cor}

\begin{proof}
Suppose without loss of generality that \(K_1 = S^3\). (If
\(K_{-1}=S^3\), then we consider the mirror knot \(\overline{K} \subset
L(p,-q)\).) If Conjecture~\ref{Conj:Num}
 is true, the corresponding simple knot 
 \(K'\) must be of either Berge or Tange type. If \(K'\) is a Berge
 knot, then \(d(K_1')=d(S^3)=0\), so \(d(K_1)=-2\). Thus \(K_1\neq
 S^3\). 
On the other hand, if \(K'\) is a Tange knot then \(K_1'\)  
is either the Poincar{\'e} sphere or its orientation reverse. To
determine the orientation, we refer to the main theorem of 
\cite{Tange2}, which says that
there is no positive surgery cobordism from \(\overline{\Sigma}\) (the
Poincare sphere oriented as 
the result of \(+1\) surgery on the positive trefoil) to a lens space.  
 Reversing the direction of the cobordism, we see that
there is no positive surgery cobordism from a lens space to
\(\Sigma\). Thus \(K_1' = \overline{\Sigma}\). 
It is well-known that  \(d(\overline{\Sigma})=-2\), 
so by the proposition \(d(K_1) = -4\). Again, we conclude that
 \(K_1\neq S^3\). 
\end{proof}

\begin{cor}
Conjecture~\ref{Conj:Num} implies Conjecture~\ref{Conj:Real}.
\end{cor}

\begin{proof} Suppose \(K \subset S^3\) has an integer lens space
  surgery \(L(p,q)\), and let \(\widetilde{K} \subset L(p,q)\) be the
  dual knot. By considering the mirror image if necessary, we may
  assume \(\widetilde{K}_1 = S^3\). By Theorem~\ref{Thrm:Simple},
  the simple knot \(K'\) in the same homology class admits an LHS
  surgery and (assuming Conjecture~\ref{Conj:Num}), the previous
  corollary implies that \(g(\widetilde{K})=g(K')<(p+1)/2\). If
 Conjecture~\ref{Conj:Num} is true, then \(K'\) is either a Berge knot
 or a Tange knot. To rule out the second possibility, observe that
 an argument very similar to the one used in the proof of
 Proposition~\ref{Prop:dInvt} shows that \(d(K'_1) =
 d(\widetilde{K}_1) = d(S^3) = 0.\) Thus \(K'\) is of Berge type, and
 \(K_1'=S^3\). By considering the dual knot, we see that \(L(p,q)\) is
 realized by surgery on a Berge knot in \(S^3\). 
\end{proof}

\begin{proof}[Proof of Corollary~\ref{Cor:L4}]
It is well known \cite{Moser} that  \(L(4n+3,4)\) may be realized as
\(4n+3\) surgery on the positive \((2,2n+1)\) torus knot. The dual
knot in \(L(4n+3,4)\) is the simple knot \(K(4n+3,4,2)\). Suppose
\(K\subset S^3\) has an integer surgery which yields \(L(4n+3,4)\), and
let \(\widetilde{K}\) be the dual knot. Then
by Theorem~\ref{Thrm:Simple}, \(\widetilde{K}\) 
is in the same homology class as a
simple knot \(K'=K(4n+3,4,k)\) which also admits an integer LHS
surgery. By Lemma~\ref{Lem:SimpleHS}, \(k^2 \equiv  \pm 4 \ (4n+3) 
\), so \(k \equiv \pm
2 \ (4n+3)\). (Since \(4n+3 \equiv 3 \ (4)\), there are no solutions with
\(k^2 \equiv -4 \ (4n+3)\).) 

From Theorem~\ref{Thrm:Simple}, we know that 
 either \(g(\widetilde{K}) = g(K')=g(T(2,2n+1))=n\), 
or \(g(\widetilde{K}) = 2n+2\). In the
 first case, Baker's theorem \cite{Baker} tells us that \(\widetilde{K}\) is a
 \((1,1)\) knot. By a theorem of Berge \cite{Berge}, this implies that
 \(\widetilde{K}\) is simple, and thus that \(\widetilde{K}=K(4n+3,4,\pm
 2)\). In the second
 case, we can apply Proposition~\ref{Prop:dInvt}
 to compute the \(d\)-invariant of
 the homology sphere \(Y\) obtained by integer surgery on
 \(\widetilde{K}\). We find
 that 
\begin{align*}
d(Y) & = d(K'_{\mp1})\pm 2 \\
& = d(S^3) \pm 2 = \pm 2
\end{align*}
so \(Y\) could not have been \(S^3\). It
 follows that
\(\widetilde{K}=K(4n+3,4,\pm 2)\), and thus that \(K\) is the positive
\((2,2n+1)\) torus knot. 
\end{proof}

We conclude by noting that knots
of the form considered in this section do exist, and that in fact 
there are infinitely many of them. 
In \cite{HeddenLens}, Hedden shows that each lens space \(L(p,q)\)
contains  a \((1,1)\) knot \(T_L\) with \(\HFK(T_L) \cong
\Z^{p+2}\)   and \(\tau(T_L, \spi_0) = -1\).
 An easy calculation shows
that \(T_L\) is in the same homology class as the simple knot
\(K(p,q,q+1)\), so it admits an integer \(\Z\)HS surgery whenever \(q
\equiv \pm (q+1)^2 \ (p)\). This surgery will be an L-space if and
only if \(q \equiv  (q+1)^2 \ (p)\) and \(\text{width} \
\HFK(K(p,q,q+1) <2p\). If we put \(k=q+1\), the first condition
becomes \(k^2-k+1 \equiv 0 \ (p)\), which implies that \(K(p,k^2,k)\)
is a Berge knot of Type VII. (See the next section for more details.)
Thus the second condition is implied by the first, and there
is exactly one knot of this type for each Berge knot of type VII. 
In \cite{HeddenLens}, Hedden conjectures that \(T_L\) and its mirror
image are the only knots in \(L(p,q)\) for which 
 \(\HFK(K) \cong
\Z^{p+2}\). If the conjecture is true, then these are the only knots
of this form. 

In small examples of this type, it is possible to identify the
resulting L-space homology sphere  as the
Poincar{\'e} sphere by using GAP \cite{GAP} to show that its
fundamental group has finite order. (The largest example for
 which the author  was able to do this was obtained by surgery on
  \(T_L \subset L(39,16)\).) It seems likely that this is always the
  case, but the author does not know how to prove it.

\section{Simple Knots}
\label{Sec:Simple}
 In this section, we explain how Conjecture~\ref{Conj:Num} can be
 rephrased as an elementary (to state, at least) question in number
 theory. We give a simple algorithm for computing the genus of the
 \(K(p,q,k)\) and explain which \(K(p,q,k)\) correspond to the knots
 found by Berge and Tange. Finally, we give some numerical evidence to
 support the conjecture.

\subsection{Genus of simple knots} 
The Fox calculus provides us with a simple algorithm
 to compute the genus of \(K(p,q,k)\). Given \(p,q\) and \(k\), we
 define a function \(f_{p,q,k} : \Z/p \to \Z\) by the relation
\begin{equation*}
f_{p,q,k}(i+1) - f_{p,q,k}(i) = \begin{cases}
  k-p & \text{if} \ iq \equiv 1,2,\ldots,k \ (p) \\ k & \text{otherwise}.
\end{cases}
\end{equation*} 
together with the normalization \(f(0) = 0\). Let \(G(p,q,k)\) be the
difference between the maximum and minimum values of \(f\). Then we
have
\begin{prop}
\(\displaystyle G(p,q,k) = \text{{\em width}} \ts \HFK(K(p,q,k)). \)
\end{prop}

\begin{proof} ({\it c.f} section 5 of \cite{OSLens}) 
We refer to the standard Heegaard diagram of
  \(K(p,q,k)\) described in section~\ref{SubSec:SimpleKnot}. Label the
  points of \(\alpha \cap \beta \) by \(x_0,x_1,\ldots,x_{p-1}\) as we go
  from left to right along \(\alpha\). As we transverse \(\beta\), we
  encounter the \(x_i\) in the following order: \(x_0,x_q,x_{2q},
  \ldots x_{(p-1)q}\). From this, it is easy to see that the relator
  corresponding to \(\beta\) is \(w=w_0w_1\ldots w_{p-1}\), where 
$$w_i = \begin{cases}
  ma & \text{if} \ iq \in [1,\ldots k] \\ a & \text{otherwise}.
\end{cases} $$
 If we
  abelianize, this relation becomes \(pa + km = 0\), so the
  abelianization map 
 is given by \(|a|=k\),
  \(|m|=-p\). 
It is now easy to see that 
\begin{equation*}
 d_a w  = \sum_{i=0}^{p-1} t ^{f_{p,q,k}(i)}
\end{equation*}
which proves the claim.
\end{proof}

\subsection{Berge knots}
Several families of simple knots with integer surgeries yielding
\(S^3\) were discovered by Berge \cite{Berge}.
 We summarize his results here.
To describe these families, 
it is enough to specify the parameters \(p\)
and \(k\), since \(q \equiv \pm k^2 \ (p)\) whenever \(K(p,q,k)\) has
an integer \(\Z\)HS surgery. 

The Berge knots may be divided into two broad classes. The first class
is much more numerous, and consists of knots in the solid torus which
have solid torus surgeries \cite{BergeTorus, GabaiTorus}. The families
in this class depend only on the value of \(p\) modulo \(k^2\). They
are
\vskip0.05in
\noindent{\bf Berge Types I and II:}
$ \displaystyle
p \equiv ik \pm 1 \ (k^2) \quad \gcd (i,k) = 1, 2 
$
\vskip0.05in
\noindent{\bf Berge Type III:}
$ \displaystyle
 \begin{cases} p \equiv \pm (2k-1)d \ (k^2), \quad d \vert k+1, \frac{k+1}{d} \
  \text{odd} \\
 p \equiv \pm (2k+1)d \ (k^2), \quad d \vert k-1, \frac{k-1}{d} \
  \text{odd}
\end{cases}
$
\vskip0.05in
\noindent{\bf Berge Type IV:}
$ \displaystyle
 \begin{cases} p \equiv \pm (k-1)d \ (k^2), \quad d \vert 2k+1, 
  \\
 p \equiv \pm (k+1)d \ (k^2), \quad d \vert 2k-1, 
\end{cases} $
\vskip0.05in
\noindent{\bf Berge Type V:}
$ \displaystyle
 \begin{cases} p \equiv \pm (k+1)d \ (k^2), \quad d \vert k+1, \ d \
  \text{odd} \\
 p \equiv \pm (k-1)d \ (k^2), \quad d \vert k-1, \ {d} \
  \text{odd}
\end{cases} 
$
\vskip0.05in
\noindent
(Berge's type VI is actually a special case of type V.)
Observe that the expressions for \(p\) in types III-V all divide
either \(2k^2\pm k -1\) (for Types III and IV) or \(k^2 \pm 2k +1\)
(for Type V). 

The remaining {\it exceptional} Berge types also involve a quadratic
expression in \(k\), but now the modulus appearing in the relation is
\(p\). They are 
\vskip0.05in
\noindent{\bf Berge Types VII and VIII:}
$ \displaystyle
k^2\pm k \pm 1 \equiv 0 \ (p) 
$
\vskip0.05in
\noindent{\bf Berge Type IX:}
$ \displaystyle
p=22j^2+9j+1, \quad \ k=11j+2 \quad  \text{for all }  j \in \Z
$
\vskip0.05in
\noindent{\bf Berge Type X:} \ 
$ \displaystyle
p=22j^2+13j+2, \quad k=11j+3 \quad  \text{for all }  j \in \Z
$
\vskip0.05in
\noindent
(Berge's types XI and XII are realized by taking negative values for
\(j\) in types IX and X respectively.) We remark that the 
knots of types IX and X all 
 satisfy the quadratic relationship \(2k^2+k+1 \equiv 0 \ (p)\). 

\subsection{Tange knots}
Examples   of simple knots
with Poincar{\'e} sphere surgeries have recently  been discovered by
Tange \cite{Tange}. He observed that with a
 single exception --- the knot \(K(191,34,15)\) --- these knots all
 fall into quadratic families, similar to Berge's families IX and X. 
Like the  exceptional Berge knots, Tange's families exhibit the
following property: the members of a given family all satisfy a simple
quadratic equation modulo \(p\). The table below lists  Tange's
 families and the quadratic relations \(Q(k) \equiv 0 \
 (p)\) 
which they satisfy.

\begin{equation*}
\begin{array}{|l|l|l||l|l|l|}
\hline
\rule{0pt}{10pt} \phantom{XXXx}p & \phantom{Xx}k & \phantom{XX}Q(k) & \phantom{XXXx}p & \phantom{Xx}k & \phantom{XX}Q(k)\\
\hline
\rule{0pt}{10pt}14j^2+7j+1 & 7j+2 & 2k^2-k+1 & & & \\
\hline
\rule{0pt}{10pt} 20j^2+15j+3 & 5j+2 & 4k^2-k-1 & & & \\
\hline
\rule{0pt}{10pt} 30j^2 +9j+1 & 6j+1 & 4k^2-k+1 & & & \\
\hline
\rule{0pt}{10pt}42j^2+23j+3 &7j+2 & 6k^2-k-1 & 42j^2+47j+13 & 7j+4 & 6k^2-k-1 \\
\hline
\rule{0pt}{10pt}52j^2+15j+1 & 13j+2 & 4k^2-k-1 & 52j^2+63j+19 & 13j +8
& 4k^2-k-1 \\
\hline
\rule{0pt}{10pt}54j^2+15j+1 & 27j+4 & 2k^2-k-1 & 54j^2+39j+7 & 27j+10
& 2k^2-k-1 \\
\hline
\rule{0pt}{10pt} 69j^2+17j+1 & 23j+3 & 3k^2-k-1 & 69j^2+29j+3 & 23j+5
& 3k^2-k-1 \\
\hline
\rule{0pt}{10pt} 85j^2+19j+1 & 17j+2 & 5k^2-k-1 & 85j^2+49j+7 & 17j+5
& 5k^2-k-1 \\
\hline
\rule{0pt}{10pt} 99j^2+35j+3 & 11j+2 & 9k^2-k-1 & 99j^2+53j+7 & 11j+3
& 9k^2-k-1 \\
\hline
\rule{0pt}{10pt} 120j^2+16j+1 & 12j+1 & 5k^2-2k+3 & 
102j^2+104j+22 & 12j+5 & 5k^2+2k-3 \\
\hline
\rule{0pt}{10pt} 120j^2+20j+1 & 20j+2 & 3k^2-2k+2 & 120j^2+36j+3 &
12j+2 & 5k^2-2k+2 \\
\hline
\end{array}
\end{equation*}
\vskip0.05in
The conjecture stated in the introduction says that \(G(p,k^2,k)<2p\)
if and only if one of the pairs   \((p,\pm k)\) or \((p, \pm k^{-1})\)
 belongs to one of the types
described in this and the preceding section. Using a computer, we have
verified that the conjecture holds for all \(p \leq 100,000\). As discussed
in the introduction, this implies that if \(L(p,q)\) is realized as
surgery on a knot in \(S^3\) for \(p \leq 100,000\), then it can be realized as
surgery on a Berge knot. 
 
\bibliographystyle{plain}
\bibliography{../mybib}

\begin{thebibliography}{10}

\bibitem{Alexander}
J.~W. Alexander.
\newblock Topological invariants of knots and links.
\newblock {\em Trans. Amer. Math. Soc.}, 30(2):275--306, 1928.

\bibitem{RolfsenLens}
J.~Bailey and D.~Rolfsen.
\newblock An unexpected surgery construction of a lens space.
\newblock {\em Pacific J. Math.}, 71(2):295--298, 1977.

\bibitem{BGH}
K.~Baker, J.E. Grigsby, and M.~Hedden.
\newblock Grid diagrams for lens spaces and combinatorial knot {F}loer
  homology.
\newblock arXiv:0710.0359, 2007.

\bibitem{Baker}
K.~L. Baker.
\newblock Small genus knots in lens spaces have small bridge number.
\newblock {\em Algebr. Geom. Topol.}, 6:1519--1621 (electronic), 2006.

\bibitem{Berge}
J.~Berge.
\newblock {Some knots with surgeries yielding lens spaces}.
\newblock unpublished manuscript.

\bibitem{BergeTorus}
J.~Berge.
\newblock The knots in {$D\sp 2\times S\sp 1$} which have nontrivial {D}ehn
  surgeries that yield {$D\sp 2\times S\sp 1$}.
\newblock {\em Topology Appl.}, 38(1):1--19, 1991.

\bibitem{Brody}
E.~J. Brody.
\newblock The topological classification of the lens spaces.
\newblock {\em Ann. of Math. (2)}, 71:163--184, 1960.

\bibitem{CrowFox}
R.H. Crowell and R.H. Fox.
\newblock {\em Introduction to Knot Theory}.
\newblock Ginn and Co., 1963.

\bibitem{CyclicS}
M.~Culler, C.~Gordon, J.~Luecke, and P.~Shalen.
\newblock Dehn surgery on knots.
\newblock {\em Ann. of Math. (2)}, 125:237--300, 1987.

\bibitem{FSLens}
R.~Fintushel and R.~Stern.
\newblock Constructing lens spaces by surgery on knots.
\newblock {\em Math. Z.}, 175(1):33--51, 1980.

\bibitem{Fox1}
R.~Fox.
\newblock Free differential calculus. {I}. {D}erivation in the free group ring.
\newblock {\em Ann. of Math. (2)}, 57:547--560, 1953.

\bibitem{Fox2}
R.~Fox.
\newblock Free differential calculus. {II}. {T}he isomorphism problem of
  groups.
\newblock {\em Ann. of Math. (2)}, 59:196--210, 1954.

\bibitem{GabaiTorus}
D.~Gabai.
\newblock Surgery on knots in solid tori.
\newblock {\em Topology}, 28(1):1--6, 1989.

\bibitem{GAP}
The GAP~Group.
\newblock {\em {GAP -- Groups, Algorithms, and Programming, Version 4.4.10}},
  2007.

\bibitem{GompfStip}
R.~E. Gompf and A.~I. Stipsicz.
\newblock {\em {$4$}-manifolds and {K}irby calculus}, volume~20 of {\em
  Graduate Studies in Mathematics}.
\newblock American Mathematical Society, Providence, RI, 1999.

\bibitem{HeddenTau}
M.~Hedden.
\newblock {An Ozsv{\'a}th-Szab{\'o} Floer homology invariant of knots in a
  contact manifold}.
\newblock arXiv: 0708.0448, 2007.

\bibitem{HeddenLens}
M.~Hedden.
\newblock {On Floer homology and the Berge conjecture on knots admitting lens
  space surgeries}.
\newblock arXiv:0710.0357, 2007.

\bibitem{KMOS}
P.~Kronheimer, T.~Mrowka, P.~Ozsv{\'a}th, and Z.~Szab{\'o}.
\newblock Monopoles and lens space surgeries.
\newblock math.GT/0310164.

\bibitem{Lipshitz}
R.~Lipshitz.
\newblock A cylindrical reformulation of {H}eegaard {F}loer homology.
\newblock {\em Geom. Topol.}, 10:955--1097 (electronic), 2006.

\bibitem{Moser}
L.~Moser.
\newblock Elementary surgery along a torus knot.
\newblock {\em Pacific J. Math.}, 38:737--745, 1971.

\bibitem{YiThesis}
Y.~Ni.
\newblock {Knot Floer homology detects fibred knots}.
\newblock math.GT/0607156.

\bibitem{YiQHS}
Y.~Ni.
\newblock {Link Floer homology detects the Thurston norm}.
\newblock math.GT/0604360.

\bibitem{OSInt}
P.~Ozsv{\'a}th and Z.~Szab{\'o}.
\newblock {Knot Floer homology and integer surgeries}.
\newblock math.GT/0410300.

\bibitem{OSFrac}
P.~Ozsv{\'a}th and Z.~Szab{\'o}.
\newblock {Knot Floer homology and rational surgeries}.
\newblock math.GT/0504404.

\bibitem{OS4}
P.~Ozsv{\'a}th and Z.~Szab{\'o}.
\newblock Absolutely graded {F}loer homologies and intersection forms for
  four-manifolds with boundary.
\newblock {\em Adv. Math.}, 173:179--261, 2003.
\newblock math.SG/0110170.

\bibitem{OS10}
P.~Ozsv{\'a}th and Z.~Szab{\'o}.
\newblock Knot {F}loer homology and the four-ball genus.
\newblock {\em Geom. Topol.}, 7:615--639, 2003.
\newblock math.GT/0301026.

\bibitem{OS7}
P.~Ozsv{\'a}th and Z.~Szab{\'o}.
\newblock Holomorphic disks and knot invariants.
\newblock {\em Adv. Math.}, 186:58--116, 2004.
\newblock math.GT/0209056.

\bibitem{OS1}
P.~Ozsv{\'a}th and Z.~Szab{\'o}.
\newblock Holomorphic disks and topological invariants for closed
  three-manifolds.
\newblock {\em Ann. of Math. (2)}, 159:1027--1158, 2004.
\newblock math.SG/0101206.

\bibitem{OSLens}
P.~Ozsv{\'a}th and Z.~Szab{\'o}.
\newblock On knot {F}loer homology and lens space surgeries.
\newblock {\em Topology}, 44:1281--1300, 2005.
\newblock math.GT/0303017.

\bibitem{OS3}
P.~Ozsv{\'a}th and Z.~Szab{\'o}.
\newblock Holomorphic triangles and invariants for smooth four-manifolds.
\newblock {\em Adv. Math.}, 202:326--400, 2006.
\newblock math.SG/0110169.

\bibitem{thesis}
J.~Rasmussen.
\newblock Floer homology and knot complements.
\newblock Harvard University thesis. math.GT/0306378, 2003.

\bibitem{Tange}
M.~Tange.
\newblock {Lens spaces given from L-space homology 3-spheres}.
\newblock arXiv:0709.0141, 2007.

\bibitem{Tange2}
M.~Tange.
\newblock {On the non-existence of lens space surgery structure}.
\newblock arXiv:0707.0197, 2007.

\bibitem{TuraevTorsion}
V.~Turaev.
\newblock {\em Torsions of {$3$}-dimensional manifolds}, volume 208 of {\em
  Progress in Mathematics}.
\newblock Birkh\"auser Verlag, Basel, 2002.

\end{thebibliography}

\end{document}